\documentclass[12pt,a4paper,twoside,reqno]{amsart}

\usepackage[all,cmtip]{xy}
\usepackage{amsmath,amscd}
\usepackage{amssymb,amsfonts}
\usepackage[driver=pdftex,margin=3cm,heightrounded=true,centering]{geometry}
\usepackage{url}
\usepackage[colorlinks=true,linkcolor=blue,citecolor=blue]{hyperref}

\usepackage{amsthm}

\theoremstyle{plain}

\newtheorem{theorem}{Theorem}[section]
\newtheorem*{thm*}{Theorem}
\newtheorem{prop}[theorem]{Proposition}
\newtheorem{lemma}[theorem]{Lemma}
\newtheorem{cor}[theorem]{Corollary}

\theoremstyle{definition}
\newtheorem{dfn}[theorem]{Definition}

\theoremstyle{remark} 

\newtheorem{example}[theorem]{Example}

\newtheorem{remark}[theorem]{Remark}

\theoremstyle{plain}

\usepackage{amscd}

\numberwithin{equation}{section}

\newcommand{\alpheqn}[1][\relax]{
     \refstepcounter{equation}
     \if#1\relax \relax
       \else \label{#1}
     \fi  
     \setcounter{saveeqn}{\value{equation}}%
    \setcounter{equation}{0}%
    \renewcommand{\theequation}{\thealphequation}}
\newcommand{\reseteqn}{\setcounter{equation}{\value{saveeqn}}%
     \renewcommand{\theequation}{\thearabicequation}}

\providecommand{\mathscr}{\mathcal} 
\IfFileExists{mathrsfs.sty}{
\usepackage{mathrsfs}}         
   {
     \IfFileExists{eucal.sty}{
        \usepackage[mathscr]{eucal}} 
   {
   }
}

\newcommand{\cd}{\cdot}

\newcommand{\ot}{\otimes}
\newcommand{\hot}{\widehat \otimes}

\newcommand{\op}{\oplus}

\newcommand{\ci}{\circ}

\newcommand{\ti}{\times}
\newcommand{\nn}{\mathbb{N}}
\newcommand{\zz}{\mathbb{Z}}

\newcommand{\rr}{\mathbb{R}}
\newcommand{\cc}{\mathbb{C}}

\newcommand{\al}{\alpha}
\newcommand{\be}{\beta}
\newcommand{\ga}{\gamma}
\newcommand{\Ga}{\Gamma}
\newcommand{\de}{\delta}

\newcommand{\ep}{\varepsilon}

\newcommand{\io}{\iota}
\newcommand{\ka}{\kappa}
\newcommand{\la}{\lambda}
\newcommand{\La}{\Lambda}

\newcommand{\Om}{\Omega}
\newcommand{\si}{\sigma}

\newcommand{\ze}{\zeta}
\newcommand{\pa}{\partial}

\newcommand{\ov}{\overline}
\newcommand{\C}[1]{\mathcal{#1}}

\newcommand{\T}[1]{\textup{#1}}

\newcommand{\E}[1]{\emph{#1}}
\newcommand{\B}[1]{\mathbb{#1}}

\newcommand{\s}[1]{\mathscr{#1}}

\newcommand{\fork}[2]{\left\{ \begin{array}{#1} #2 \end{array} \right.}

\newcommand{\rar}{\Rightarrow}

\newcommand{\su}{\subseteq}

\newcommand{\q}{\qquad}

\newcommand{\inn}[1]{\langle #1 \rangle}

\newcommand{\cro}[1]{C_r^*(\zz,#1)}

\newcommand{\sem}{\setminus}

\newtheorem{thm}[theorem]{Theorem}
\newtheorem{lem}[theorem]{Lemma}
\usepackage{graphicx}

\newcommand{\sa}{{\operatorname{sa}}}
\newcommand{\ZZ}{{\mathbb{Z}}}

\newcommand{\vertiii}[1]{{\left\vert\kern-0.25ex\left\vert\kern-0.25ex\left\vert #1 
    \right\vert\kern-0.25ex\right\vert\kern-0.25ex\right\vert}}
\newcommand{\Bvert}[1]{{\Big\vert\kern-0.25ex\Big\vert\kern-0.25ex\Big\vert #1 
    \Big\vert\kern-0.25ex\Big\vert\kern-0.25ex\Big\vert}}
\newcommand{\bvert}[1]{{\big\vert\kern-0.25ex\big\vert\kern-0.25ex\big\vert #1 
    \big\vert\kern-0.25ex\big\vert\kern-0.25ex\big\vert}}
\newcommand{\nvert}[1]{{\vert\kern-0.25ex\vert\kern-0.25ex\vert #1 
    \vert\kern-0.25ex\vert\kern-0.25ex\vert}}

\usepackage{a4wide}
\renewcommand{\leq}{\leqslant}
\renewcommand{\geq}{\geqslant}

\usepackage{thmtools}
\declaretheorem[style=theorem,name={Theorem}]{theoremletter}

\declaretheorem[style=theorem,name={Question}]{questionletter}
\newtheorem{introquestion}[questionletter]{Question}

\begin{document}

\author{Jens Kaad}
\address{Jens Kaad, Department of Mathematics and Computer Science, University of Southern Denmark, Campusvej 55, DK-5230 Odense M, Denmark}
\email{kaad@imada.sdu.dk}

\author{David Kyed}
\address{David Kyed, Department of Mathematics and Computer Science, University of Southern Denmark, Campusvej 55, DK-5230 Odense M, Denmark}
\email{dkyed@imada.sdu.dk}

\dedicatory{Dedicated to Ryszard Nest on the occasion of his upcoming retirement \\   --- with great admiration and gratitude for all the things you taught us.}

\subjclass[2010]{ Primary: 46L05; Secondary: 46L89, 47L65,  46L05 }
\keywords{Quantum metric spaces, crossed products, dynamics of Riemannian manifolds}

\begin{abstract}
We provide a detailed study of actions of the integers on compact quantum metric spaces, which includes general criteria ensuring that the associated crossed product algebra is again a compact quantum metric space in a natural way. We moreover provide a flexible set of assumptions ensuring that a continuous family of $*$-automorphisms of a compact quantum metric space, yields a field of crossed product algebras which varies continuously in Rieffel's quantum Gromov-Hausdorff distance. Lastly we show how our results apply to continuous families of Lip-isometric actions on compact quantum metric spaces and to families of diffeomorphisms of compact Riemannian manifolds which vary continuously in the Whitney $C^1$-topology.
\end{abstract}

\title{Dynamics of compact quantum metric spaces}

\maketitle

\section{Introduction}

In a series of papers from the late nineties and the beginning of this century, a theory of quantum metric spaces was initiated and developed by Rieffel, \cite{Rie:MSS,Rie:MSA,Rie:GHD}. The core idea is to investigate metrics on state spaces of $C^*$-algebras that arise via duality from a densely defined semi-norm (called the Lip-norm) with the crucial property being that the metric in question actually metrizes the weak $*$-topology. The classical origin of this idea is the Monge-Kantorovi\v{c} metric on the set of regular Borel probability measures on a compact metric space, \cite{Kan:ESQ}, but Rieffel's theory is also deeply linked to current developments in noncommutative geometry, \cite{Con:CFH,Con:NCG,Con:GFN}. Indeed, the prototypical example of a semi-norm arises as the composition of the operator norm with the densely defined derivation given by taking commutators with  the Dirac operator of a spectral triple, \cite{Rie:CQM}. \\
A particularly interesting aspect of Rieffel's theory is that it admits a noncommutative counterpart of the Gromov-Hausdorff distance for compact metric spaces, \cite{Gro:GPE,Rie:GHD}. This notion of distance is referred to as the quantum Gromov-Hausdorff distance and one may apply it to investigate convergence questions for families of quantum metric spaces in a whole range of contexts. Most famously, this was addressed by Rieffel in \cite{Rie:MSG} for the case of matrices (fuzzy spheres) converging to the $2$-sphere and for noncommutative tori, which vary continuously in the deformation parameter $\theta$ (a real skew-symmetric matrix), see \cite{Rie:GHD}. This  line of reasoning was further developed by Li who proved that more general $\theta$-deformations of torodial spin manifolds vary continuously in the deformation parameter when the semi-norm comes from the Connes-Landi-Dubois-Violette Dirac operator, see \cite{Li:DCQ,Li:GH-dist,CoLa:NII,CoDu:NFS}. Further important work on the convergence problem was carried out by Latr\'emoli\`ere, in particular regarding the   approximations of tori by fuzzy tori \cite{Lat:AQQ}.
In more recent years, new developments were pioneered by Latr\'emoli\`ere on alternative distances between quantum metric spaces, see \cite{Lat:DGH, Lat:QGH},  which has given rise to numerous new convergence results, including  the approximation of non-commutative solenoids by quantum tori by Latr\'emoli\`ere and Packer \cite{LatPack:solenoids} and the approximation of AF-algebras by matrix algebras by Aguilar and Latr\'emoli\`ere   \cite{LatAgu:AF}. A pivotal role in the deep work of Latr\'emoli\`ere is played by the Leibniz inequality, which is for example satisfied when the semi-norm comes from a derivation. One of the main motivations for this work was to incorporate the $*$-algebraic structure of $C^*$-algebras and not only the order unit structure of their real part. Indeed, in Latr\'emoli\`ere's framework it holds that two $C^*$-algebraic quantum metric spaces with distance equal to zero are actually $*$-isomorphic and not just isomorphic as order unit spaces. It should be mentioned that this property is also shared by other notions of distance as proposed in \cite{Li:CQG,Ker:MQG}. \\

In the present work we shall consider the general setup of a $C^*$-algebraic compact quantum metric space $(B,L)$ equipped with a $*$-automorphism $\be \colon B \to B$. 
The dynamics of $\beta$ is naturally encoded in the associated crossed product algebra  $\cro{B,\be}$, and our main focus will be on lifting the quantum metric structure from $B$ to  $\cro{B,\be}$ and investigate convergence problems in the quantum Gromov-Hausdorff metric under deformations of the $*$-automorphism $\be$.
More specifically we are interested in the following two central questions:
\begin{introquestion}\label{que:first-question}
Suppose that $B$ carries a densely defined semi-norm $L_B \colon V_B \to [0,\infty)$ such that $(B,L_B)$ becomes a compact quantum metric space. When can we ensure that the crossed product $\cro{B,\be}$ becomes a compact quantum metric space for some `explicit and natural' semi-norm $L$ involving the semi-norm $L_B$?
\end{introquestion}
Perhaps a comment pertaining to the adjectives `explicit' and `natural' is in order: Similarly to the classical result that any separable compact Hausdorff space is metrizable, one can also show that any separable unital $C^*$-algebra can be endowed with a compact quantum metric space structure \cite{Rieffel:group-C-star-algebras-as-cqms}. But, as is the case in the classical situation, this quantum metric structure is not necessarily particularly illuminating. {For the case of crossed products referred to in the above question,} it would for example not reflect the additional data coming from the fixed semi-norm on $B$ and the automorphism $\beta$, so in this way such an answer is not very interesting and we are seeking more naturally occurring seminorms which capture the additional data available.\\
Provided that one obtains a satisfying answer to Question \ref{que:first-question}, the following problem arises:

\begin{introquestion}\label{que:second-question}
Suppose that one has a parametrized family of $*$-automorphisms $\{\beta_t\}_{t\in T}$ of the base algebra $B$ and a way of turning the associated crossed product algebras $\{ \cro{B,\be_t}\}_{t \in T}$ into compact quantum metric spaces. Under which conditions on the family $\{\beta_t\}_{t\in T}$ do these vary continuously in the quantum Gromov-Hausdorff distance?
\end{introquestion}


Question \ref{que:first-question} was previously investigated in \cite{BMR:DSS} and \cite{HSWZ:STC}. The main focus in these papers was on the situation where the semi-norm $L_B \colon V_B \to [0,\infty)$ comes from a unital spectral triple $(\s B,H,D)$ by taking commutators, so that $V_B = \s B$ and $L_B(b) := \| \ov{[D,b]} \|_\infty$, for all $b \in \s B$. Moreover, most of the work is carried out under the additional assumption of equicontinuity meaning that $\sup_{n \in \zz} L_B(\be^n(b)) < \infty$, for all $b \in \s B$, implicitly saying here also that $\be(\s B) = \s B$. This assumption is for example satisfied when $\be$ is a Lip-isometry in the sense that $L_B(\be(b)) = L_B(b)$, for all $b \in \s B$. In the equicontinuous setting it is possible to write down a semi-norm, which turns the crossed product  $\cro{B,\be}$ into a compact quantum metric space, and  comes from a unital spectral triple for the crossed product algebra. When leaving the equicontinuous setting, not much seems to be known, but a general result was obtained in this direction by Bellissard, Marcolli, and Reihani in \cite[Theorem 3]{BMR:DSS}, using ideas inspired and motivated by the constructions of Connes and Moscovici for crossed products by groups of diffeomorphisms, \cite{CoMo:LIF}. However, a serious drawback of the construction in \cite[Theorem 3]{BMR:DSS} is that one is forced to enlarge to a non-unital base algebra before performing the crossed product construction, corresponding to the passage from a compact to a non-compact base space.
The methods developed in the present paper also apply beyond the equicontinuous setting but are different from what has previously been considered, in so far that we do not need to change the $C^*$-algebra and we may thus work directly with the crossed product in question. In particular, contrary to what happens in \cite{BMR:DSS}, we do not need to pass to locally compact quantum metric spaces, which we consider to be an important point since locally compact quantum metric spaces are known to display quite a different behaviour than their compact counterparts, \cite{Lat:BLD,Lat:QLC,MeRe:NMU}. The price to pay, however, is that our compact quantum metric spaces may not satisfy the Leibniz inequality. \\

Our first main result provides the following general answer to Question \ref{que:first-question}:


\begin{theoremletter}[see Theorem \ref{t:quacross}]\label{thm:intro-theorem-on-cqms-structure}
Let $(B,L_B)$ be a $C^*$-algebraic compact quantum metric space and suppose that the semi-norm $L_B$ is lower semi-continuous on its  domain $V_B$ and that $\be(V_B) = V_B$. Let $\nvert{\cd} : C_c(\zz) \to [0,\infty)$ be a norm which is order preserving in the sense that $\bvert{ h } \leq \bvert{ k }$ whenever $0\leq h(n) \leq k(n)$ for all $n \in \zz$. Then the semi-norm $L \colon C_c(\zz,V_B) \to [0,\infty)$ defined by
\[
L( x ) := \max\left\{ \Big\| \sum_n n x(n) U^n \Big\| , \bvert{ L_B \ci x}, \bvert{L_B \ci x^*} \right\}
\]
provides the crossed product $\cro{B,\be}$ with the structure of a $C^*$-algebraic compact quantum metric space. 
\end{theoremletter}

We consider it to be a rather striking feature of the above theorem that the only link between the semi-norm $L_B \colon V_B \to [0,\infty)$ and the $*$-automorphism $\be : B \to B$ is that both $\be$ and $\be^{-1}$ are required to preserve the domain $V_B$, so that we do not need to link $\be$ to $L_B$ by imposing extra continuity constraints. Remark in this respect that the domain $V_B$ need not be complete with respect to the norm $\| \cd \| + L_B(\cd)$ so that one can not use the closed graph theorem to obtain automatic continuity results. \\ 

In the geometric setting, where the unital base $C^*$-algebra $B$ is equipped with a unital spectral triple $(\s B,H,D)$ and $L_B(b):=\|\overline{[D,b]}\|_\infty$ we also provide a more geometric construction of semi-norms relating more directly to the spectral geometry of the crossed product. We emphasize that the following theorem, for instance, applies to the general setting where $M$ is a connected, compact, Riemannian manifold equipped with \emph{any} diffeomorphism $\psi \colon M \to M$. In this context, we may choose $D \colon \T{Dom}(D) \to L^2(M,\La T^* M)$ to be the closure of the Hodge-de Rham operator $d + d^*$ and $\be \colon C(M) \to C(M)$ comes from the diffeomorphism via Gelfand duality: $\be(f) := f \ci \psi$. The constants $\la,\mu \geq 1$ appearing in the theorem below then correspond to the operator norms of the exterior derivatives $d \psi$ and $d \psi^{-1}$.

\begin{theoremletter}[see Corollary \ref{cor:BMR-cor}]\label{t:twisted}
Suppose that $(\s B,H,D)$ is a unital spectral triple of parity $p \in \{ \T{even}, \T{odd} \}$, which is also a spectral metric space in the sense that the semi-norm $L_B(b) := \| \ov{[D,b]} \|_\infty$ provides the $C^*$-completion $B$ with the structure of a compact quantum metric space. Suppose, moreover, that $\be(\s B) = \s B$ and that there exist constants $\la,\mu \geq 1$ such that
\[
\sup_{n \in \nn_0} \mu^{-n} \cd L_B( \be^n(b)) < \infty \q \mbox{and} \q \sup_{n \in \nn_0} \la^{-n} \cd L_B(\be^{-n}(b)) < \infty,
\]
for all $b \in \s B$. Then there exist a $*$-derivation $\pa_1$ and a twisted $*$-derivation $\pa_\Ga$, both defined on $C_c(\zz,\s B)$, such that the semi-norm 
\[
L_\Ga(x) := \fork{ccc}{ \max\Big\{ \| \pa_\Ga(x) + i \pa_1(x) \|_\infty, \| \pa_\Ga(x) - i \pa_1(x) \|_\infty \Big\} & \E{for} & p \E{ odd } \vspace{0.2cm} \\
\| \pa_\Ga(x) + (1 \ot \ga) \pa_1(x) \|_\infty  & \T{for} & p \E{ even},
} 
\]
provides the crossed product $\cro{B,\be}$ with the structure of a compact quantum metric space. Here, as usual, $\gamma$ denotes the grading on $H$ in the even case.
\end{theoremletter}

Theorem \ref{t:twisted} contains a result of Bellissard, Marcolli, and Reihani as a special case corresponding to the equicontinuous setting where one may choose $\la = \mu = 1$, \cite[Proposition 3]{BMR:DSS}; see also \cite[Theorem 2.11]{HSWZ:STC}. In the general setting of Theorem \ref{t:twisted}, the candidate for a unital spectral triple provided by Bellissard, Marcolli, and Reihani, does not have bounded commutators with the elements in the coordinate algebra $C_c(\zz,\s B)$, and can therefore not be used to construct a reasonable seminorm on the crossed product algebra.
We resolve this problematic behaviour by ``rescaling'' the horizontal part of the geometry using powers of the constants $\la,\mu \geq 1$, and this rescaling procedure is in turn responsible for the twisting of the $*$-derivation $\pa_\Ga$. The rescaling procedure we consider here is inspired by the modular techniques developed in \cite{Kaa:UKM} in relation to the unbounded Kasparov product. In general it seems that only little is known on the structure of noncommutative geometries arising from a pair of unbounded selfadjoint operators, where one of them yields an ordinary derivation but the other one only yields a twisted derivation. The reader may, however, consult the paper \cite{KRS:RFH} where a similar structure is encountered. For other interesting approaches to the noncommutative geometry of crossed products (and many other things) we refer to \cite{PoWa:NGC,IoMa:CES,GMR:UTS} and of course \cite{CoMo:LIF,CoMo:TST,Mos:LIT}. These papers are not essential for the understanding of the present text but the references are included to put the present paper into a broader context. \\

Regarding Question \ref{que:second-question}, as far as the authors are aware, not much is known if we disregard the extensively studied noncommutative $2$-tori, which can indeed be viewed as crossed products by the integers. Notably, we have not even been able to find general results regarding the convergence problem for Lip-isometric actions. Our next main result fills this gap by providing an answer to Question \ref{que:second-question} in the Lip-isometric case.

\begin{theoremletter}[see Corollary \ref{c:genisomet}]\label{t:isometric-intro-theorem}
Assume that $(B,L_B)$ is a $C^*$-algebraic compact quantum metric space with $L_B$ lower semi-continuous on its domain $V_B$, and that $\{\beta_t\}_{t\in T}$ is a family of $*$-automorphisms of $B$ parametrized by a compact Hausdorff space $T$.
Suppose moreover that the map $t \mapsto \be_t(b)$ is continuous for all $b \in B$ and that the $*$-automorphisms $\be_t \colon B \to B$ satisfy $\beta_t(V_B)=V_B$ and are Lip-isometric in the sense that $L_B(\be_t(b)) = L_B(b)$ for all $b \in V_B$ and all $t \in T$. 
Then for any choice of order preserving norm $\vertiii{\cdot}$ on $ C_c(\zz)$, the seminorm $L_t$ provided by Theorem \ref{thm:intro-theorem-on-cqms-structure} turns $ \cro{B,\be_t}$ into a compact quantum metric space and the family
 $\big\{ ( \cro{B,\be_t}, L_t) \big\}_{t \in T}$ varies continuously in the quantum Gromov-Hausdorff distance.
\end{theoremletter}

In fact,  in Theorem \ref{t:crosconv} we provide a general criterion for convergence in quantum Gromov-Hausdorff distance which is formulated in terms of the fixed semi-norm $L_B \colon V_B \to [0,\infty)$ and the strongly continuous family of $*$-automorphisms $\be \colon T \to \T{Aut}(B)$, from which  Theorem \ref{t:isometric-intro-theorem} can be derived as a corollary. This more general criterion builds on  a somewhat delicate blend of uniformity criteria studied by Rieffel \cite{Rie:GHD}  and Li \cite{Li:GH-dist},  which we develop in Section \ref{s:unicri}, and applies also beyond the Lip-isometric case. As an illustration of this flexibility in our theory, we also obtain the following continuity result for non-Lip-isometric actions on Riemannian manifolds:

\begin{theoremletter}\label{t:riemannian-into-theorem}
Suppose that $M$ is a connected, compact, Riemannian manifold and equip $C(M)$ with the densely defined, lower semi-continuous semi-norm $L_{C(M)} \colon C^\infty(M) \to [0,\infty)$ given by $L_{C(M)}(f) := \sup_{p \in M} \| df(p) \|_\infty$, where $d \colon C^\infty(M) \to \Ga^\infty(T^* M)$ denotes the exterior derivative. Suppose, moreover, that $\psi \colon T \to \T{Diff}(M)$ is a family of diffeomorphisms of $M$, parametrized by a compact Hausdorff space $T$, which is continuous with respect to the Whitney $C^1$-topology, and that $\vertiii{\cd} : C_c(\zz) \to [0,\infty)$ is an order preserving norm. For each $t \in T$, define the $*$-automorphism of $C(M)$ by $\be_t(f) := f \ci \psi_t$ and denote by $L_t$ the densely defined semi-norm on $\cro{B,\beta_t}$ provided by Theorem \ref{thm:intro-theorem-on-cqms-structure}.  Then $\big\{(\cro{B,\be_t},L_t)\big\}_{t \in T}$ is a family of compact quantum metric spaces  which varies continuously in the quantum Gromov-Hausdorff distance.
\end{theoremletter}

We emphasize that we are considering Rieffel's original notion of quantum Gromov-Hausdorff distance since our compact quantum metric spaces may not satisfy the Leibniz inequality and therefore do not fit within Latr\'emoli\`ere's framework, this is for example the case for the Lip-norms arising in Theorem \ref{t:twisted} for $\la$ or $\mu > 1$. 
However, as pointed out to us by Latr\'emoli\`ere, one may enlarge the semi-norm so that it also controls the difference between the identity and the twists,  thereby forcing it to satisfy a quasi-Leibniz inequality; see e.g.~\cite{Lat:Compactness-theorem}. It seems to be an interesting problem to extend Latr\'emoli\`ere's ideas on quantum distances so that they can apply directly to the kind of compact quantum metric spaces arising from the combinations of $*$-derivations and twisted $*$-derivations considered in Theorem \ref{t:twisted}. This kind of question does however lie beyond the scope of the present text.
It also remains an open problem to investigate convergence in the quantum Gromov-Hausdorff distance with respect to the semi-norms described in Theorem \ref{t:twisted}. This seems to be a rather challenging task from an analytic perspective. \\

Let us end this introduction by saying a few words on the techniques that we apply in order to prove our main theorems. Theorem \ref{thm:intro-theorem-on-cqms-structure} and \ref{t:twisted}, concerning compact quantum metric space structures on crossed products, are derived using a general method which is due to Li, \cite{Li:DCQ}, relating to earlier results on ergodic actions by Rieffel, \cite{Rie:MSA,Rie:GHD}. The proofs of Theorem \ref{t:isometric-intro-theorem} and \ref{t:riemannian-into-theorem}, regarding  convergence in the quantum Gromov-Hausdorff distance, are inspired by ideas of Li on continuous fields of compact quantum metric spaces. As already indicated above, our general criterion, Theorem \ref{t:GHconv}, for convergence is related to (but different from) Li's uniform criterion for convergence, which can be found in \cite{Li:GH-dist}. In fact, we are in some sense mixing parts of Li's uniform criterion with ideas of Rieffel applied to finite dimensional compact quantum metric spaces in \cite{Rie:GHD}.

\subsection*{Acknowledgments}
The authors gratefully acknowledge the financial support from  the Villum Foundation (grant no.~7423) and from the Independent Research Fund Denmark (grant no.~7014-00145B and grant no.~9040-00107B). The authors would also like to thank Fr{\'e}d{\'e}ric Latr\'emoli\`ere and Hanfeng Li for their illuminating remarks on an earlier version of the paper. Finally, we would like to thank the anonymous referee for many nice comments improving the quality of the paper.

\section{Preliminaries on compact quantum metric spaces}
Let $A$ be a unital $C^*$-algebra and consider a complex vector subspace $V \su A$ with $1 \in V$ and with $\xi^* \in V$ for all $\xi \in V$. {We emphasize that $V \su A$ need not be dense in norm.} The real part of $V$ is denoted by
\[
V^{\T{sa}} := \{ \xi \in V \mid \xi = \xi^* \}  \su A^{\T{sa}} .
\]
The partial order $\leq$ on $A^{\T{sa}}$ induces a partial order $\leq$ on $V^{\T{sa}}$ and this partial order together with the unit $1 \in V^{\T{sa}}$ provides $V^{\T{sa}}$ with the structure of an \emph{order unit space}, see \cite{Kad:RTC},  and the order unit norm agrees with the norm inherited from the $C^*$-algebra $A$. The \emph{state space} of $V^{\T{sa}}$ is defined by
\[
S(V^{\T{sa}}) := \{ \mu : V^{\T{sa}} \to \rr \mid \mu \T{ is $\rr$-linear and positive with } \mu(1) = 1 \} .
\]
The state space is equipped with the weak $*$-topology (i.e.~the topology of pointwise convergence) and is a compact topological space with respect to this topology. We remark that the theory of order unit spaces can be developed from an abstract point of view, but in the applications we have in mind the order unit spaces will always be concrete in the sense that they are realised inside a specific unital $C^*$-algebra. Consider now a semi-norm $L^{\T{sa}} \colon  V^{\T{sa}} \to [0,\infty)$ with $L^{\T{sa}}(1) = 0$ and define the map $\rho_{L^{\T{sa}}} \colon  S(V^{\T{sa}}) \ti S(V^{\T{sa}}) \to [0,\infty]$ by
\[
\rho_{L^{\T{sa}}}(\mu,\nu) := \sup\{ | \mu(\xi) - \nu(\xi) | \mid \xi \in V^{\T{sa}}, L^{\T{sa}}(\xi) \leq 1 \} .
\]
The map $\rho_{L^{\T{sa}}}$ satisfies all the requirements of a metric on $S(V^{\T{sa}})$ except that we allow $\rho_{L^{\T{sa}}}$ to take the value $\infty$. With the notation just introduced, we are now able to present Rieffel's definition of a non-commutative analogue of a compact metric space:

\begin{dfn}\label{d:oucpms}
The pair $(V^{\T{sa}},L^{\T{sa}})$ is called an \emph{order unit compact quantum metric space}  if  $L^{\T{sa}}(1) = 0$ and $\rho_{L^{\T{sa}}} \colon  S(V^{\T{sa}}) \ti S(V^{\T{sa}}) \to [0,\infty]$ metrizes the weak $*$-topology on $S(V^{\T{sa}})$.
\end{dfn}
Note that since $S(V^{\T{sa}})$ is weak $*$-compact { and connected}, it is implicitly contained in the definition of an order unit compact quantum metric space that the metric $\rho_{L^{\T{sa}}}$ is actually finite. It is usually part of the definition of an order unit compact quantum metric space that $\T{Ker}(L^{\T{sa}}) = \rr 1$, see for example \cite[Definition 2.2]{Rie:GHD}. We now argue that it suffices to assume that $\rr 1 \su \T{Ker}(L^{\T{sa}})$, which is indeed part of our  assumptions on an order unit compact quantum metric space. 

\begin{lemma}\label{l:qmsker}
Suppose that $(V^{\T{sa}},L^{\T{sa}})$ is an order unit compact quantum metric space, then $\T{Ker}(L^{\T{sa}}) = \rr 1$. 
\end{lemma}
\begin{proof}
Since $L^{\T{sa}}(1) = 0$ by assumption we have that $\rr 1 \su \T{Ker}(L^{\T{sa}})$. Suppose now, for contradiction, that $(V^{\T{sa}},L^{\T{sa}})$ is a compact quantum metric space and that we have an element $x \in \T{Ker}(L^{\T{sa}}) \su V^{\T{sa}}$ with $x \notin \rr 1$. Since the state space $S(V^{\T{sa}})$ is compact in the weak $*$-topology we must have that the metric $\rho_{L^{\T{sa}}} \colon  S(V^{\T{sa}}) \ti S(V^{\T{sa}}) \to [0,\infty]$ is bounded so that there exists a $d \in [0,\infty)$ with $\rho_{L^{\T{sa}}}(\mu,\nu) \leq d$ for all $\mu,\nu \in S(V^{\T{sa}})$. But since $x \notin \cc 1$ (because $x = x^*$ and $x \notin \rr 1$) we can find two states $\mu_0,\nu_0 : A \to \cc$ such that $\mu_0(x) \neq \nu_0(x)$, see \cite[Theorem 4.3.4(i)]{KaRi:FTO}. This implies that $\rho_{L^{\T{sa}}}(\mu_0|_{V^{\T{sa}}},\nu_0|_{V^{\T{sa}}}) = \infty$ which contradicts the inequality $\rho_{L^{\T{sa}}}(\mu_0|_{V^{\T{sa}}},\nu_0|_{V^{\T{sa}}}) \leq d$. 
\end{proof}

We remark that the above proof only uses the assumption that the metric $\rho_{L^{\T{sa}}} \colon  S(V^{\T{sa}}) \ti S(V^{\T{sa}}) \to [0,\infty]$ is bounded.

\begin{dfn}
Suppose that $(V^{\T{sa}},L^{\T{sa}})$ is an {order unit compact quantum metric space}. We define the \emph{radius} of $(V^{\T{sa}},L^{\T{sa}})$ as the radius of the state space $S(V^{\T{sa}})$ with respect to the metric $\rho_{L^{\T{sa}}}$. To wit, this is the quantity 
\[
\frac{1}{2}\sup\big\{ \rho_{L^{\T{sa}}}( \mu,\nu) \mid \mu,\nu \in S(V^{\T{sa}}) \big\} .
\]
\end{dfn}

The next result can be found as \cite[Proposition 2.2]{Rie:MSS}.

\begin{prop}\label{p:radbou}
Let $r \geq 0$ and suppose that $\T{Ker}(L^{\T{sa}}) = \rr 1$. The following conditions are equivalent:
\begin{enumerate}
\item $\rho_{L^{\T{sa}}}(\mu,\nu) \leq 2r$ for all $\mu,\nu \in S(V^{\T{sa}})$;
\item $\inf_{\la \in \rr} \| \xi + \la 1 \| \leq r \cd L^{\T{sa}}(\xi)$ for all $\xi \in V^{\T{sa}}$.
\end{enumerate}
\end{prop}

Suppose now that $L^{\T{sa}} \colon  V^{\T{sa}} \to [0,\infty)$ is the restriction of a semi-norm $L \colon  V \to \cc$ satisfying that $L(\xi) = L(\xi^*)$ for all $\xi \in V$. Suppose moreover that $V \su A$ is dense in $C^*$-norm.
We define a map $\rho_L \colon  S(A) \ti S(A) \to [0,\infty]$ on the state space $S(A)$ of $A$ by
\[
\rho_L(\mu,\nu) := \sup \{ | \mu(\xi) - \nu(\xi) | \mid \xi \in V, L(\xi) \leq 1 \} .
\]
As above, $\rho_L$ satisfies all the requirements of a metric except that $\rho_L$ may take the value $\infty$. Since the involution is an isometry for the semi-norm $L \colon  V \to [0,\infty)$ we have that
\[
\rho_L(\mu,\nu) = \sup \{ | \mu(\xi) - \nu(\xi) | \mid \xi \in V^{\T{sa}} , L(\xi) \leq 1 \} = \rho_{L^{\T{sa}}}(\mu|_{V^{\T{sa}}}, \nu|_{V^{\T{sa}}}).
\]
Indeed, for $\xi \in V$ with $L(\xi) \leq 1$ and for $\mu,\nu \in S(A)$, choose a $\la \in S^1$ such that $\la \cd (\mu(\xi) - \nu( \xi) ) \in \rr$. Then it holds that $\T{Re}(\la \cd \xi) = (\la \cd \xi + \ov{\la} \cd \xi^*) / 2 \in V^{\T{sa}}$, that $L( \T{Re}(\la \xi)) \leq 1$ and that
\[
| \mu( \T{Re}(\la \cd \xi)) - \nu( \T{Re}(\la \cd \xi)) | = | \mu( \la \cd \xi) - \nu( \la \cd \xi)| = |\mu(\xi) - \nu(\xi)| .
\]

The next definition coincides with Li's definition of a $C^*$-algebraic compact quantum metric space, see for example \cite[Definition 2.3]{Li:ECQ}.


\begin{dfn}\label{d:cqms}
Suppose that $L(\xi) = L(\xi^*)$ for all $\xi \in V$, that $L(1) = 0$, and that the $*$-invariant unital subspace $V \su A$ is dense in $C^*$-norm. We then say that the pair $(A,L)$ is a \emph{compact quantum metric space} when the metric $\rho_L$ metrizes the weak $*$-topology on $S(A)$.
\end{dfn}

Notice that the proof of Lemma \ref{l:qmsker} can be adapted to the complex setting showing that if $(A,L)$ is a compact quantum metric space, then $\T{Ker}(L) = \cc 1$.
We denote the quotient map by $q \colon  A \to A / \cc 1$ and the quotient norm by $\| \cd \|_{A/ \cc 1} \colon  A / \cc 1 \to [0,\infty)$. We apply the same notation for the quotient map on the selfadjoint part $q \colon  A^{\T{sa}} \to A^{\T{sa}} / \rr 1$ and on this quotient space we have the quotient norm $\| \cd \|_{A^{\T{sa}} / \rr 1}$. For each $R \geq 0$, we define
\begin{equation}\label{eq:balls}
\begin{split}
 \B B_R(L) &:= \{ \xi \in V \mid L(\xi) \leq R \}, \\ 
 \B B_R(L^{\T{sa}}) &:= {\B B_R(L)} \cap A^{\T{sa}} := \{ \xi \in V^{\T{sa}} \mid L^{\T{sa}}(\xi) \leq R \} .
\end{split}
\end{equation}
Before stating the next result, we remind the  reader that a subset of a metric space is said to be \emph{totally bounded} if it can be covered by finitely many $\ep$-balls for any $\ep>0$, and that this notion is equivalent with having compact closure when the ambient metric space is complete.

\begin{lemma}\label{l:totker}
If $q(\B B_1(L) ) \su A / \cc 1$ is totally bounded with respect to the quotient norm $\| \cd \|_{A / \cc 1} \colon  A / \cc 1 \to [0,\infty)$, then $\T{Ker}(L) = \cc 1$. A similar result holds with $L$ replaced by $L^{\T{sa}}$ and $\cc 1$ replaced by $\rr 1$.
\end{lemma}
\begin{proof}
We only give the proof in the complex case, the real case being similar. Since $L(1) = 0$ by assumption we have that $\cc 1 \su \T{Ker}(L)$. We now proceed by contraposition. So suppose that $x \in \T{Ker}(L)$ and that $x \notin \cc 1$. Then we have an isometry $\io \colon  \cc \to q(\B B_1(L) )$ defined by $\io(\la) = \la \cd q(x) / \| q(x) \|_{A / \cc 1}$ for all $\la \in \cc$. This implies that $q(\B B_1(L) )$ is not totally bounded.
\end{proof}

Using Lemma \ref{l:qmsker} and Lemma \ref{l:totker}, we may now quote the following result from \cite[Theorem 1.8]{Rie:MSA}:

\begin{thm}\label{t:quamet}
Suppose that $L(\xi) = L(\xi^*)$ for all $\xi \in V$, that $L(1) = 0$ and that the $*$-invariant unital complex subspace $V \su A$ is dense in $C^*$-norm. Then the following statements are equivalent:
\begin{enumerate}
\item $(A,L)$ is a compact quantum metric space;
\item $(V^{\T{sa}},L^{\T{sa}})$ is an order unit compact quantum metric space;
\item $q(\B B_1(L) ) \su A / \cc 1$ is totally bounded with respect to the quotient norm $\| \cd \|_{A/\cc 1}$;
\item  $q(\B B_1(L^{\T{sa}})) \su A^{\T{sa}} / \rr 1$ is totally bounded with respect to the quotient norm $\| \cd \|_{A^{\T{sa}}/\rr 1}$.
\end{enumerate}
\end{thm}

In what follows we will almost always, and often without reference, be using the criteria (2) and (3) when arguing that  something is a compact quantum metric space.

\section{Quantum Gromov-Hausdorff distance}\label{s:ghdist}
In this section, we recall Rieffel's and Li's notions of distance between order unit compact quantum metric spaces introduced in \cite{Rie:GHD} and \cite{Li:GH-dist}, respectively. Both notions build on the following classical notion of Hausdorff distance:

\begin{dfn}\label{d:hausdorff}
For a (not necessarily compact) metric space $(Z,\rho)$ and two totally bounded subsets $X, Y \su Z$,  the \emph{Hausdorff distance} $\T{dist}^\rho_H(X,Y) \in [0,\infty)$ is defined by
\[
\T{dist}^\rho_H(X,Y) := \inf\big\{ r \geq 0 \mid \rho(x,Y) \leq r , \forall  x \in X  
\, \, \T{ and } \, \, \, \rho(y,X) \leq r , \forall y \in Y \big\} .
\]
\end{dfn}

Let $(V_1^{\T{sa}},L_1^{\T{sa}})$ and $(V_2^{\T{sa}},L_2^{\T{sa}})$ be order unit compact quantum metric spaces. We denote the corresponding unital complex vector subspaces by $V_1 \su A_1$ and $V_2 \su A_2$ (so $A_1$ and $A_2$ are unital $C^*$-algebras). The direct sum $V_1^{\T{sa}} \op V_2^{\T{sa}}$ is then again an order unit space, arising as the real part of the complex vector subspace $V_1 \op V_2 \su A_1 \op A_2$. Using the projections $\pi_j \colon  V_1^{\T{sa}} \op V_2^{\T{sa}} \to V_j^{\T{sa}}$ for $j = 1,2$ we may consider the state spaces $S(V_j^{\T{sa}})$ as compact subsets of the state space $S(V_1^{\T{sa}} \op V_2^{\T{sa}})$ and therefore, for any metric $\rho$ on $S(V_1^{\T{sa}} \op V_2^{\T{sa}})$, which metrizes the weak $*$-topology, we have the Hausdorff distance
\[
\T{dist}_H^{\rho}\big(  S(V_1^{\T{sa}}), S(V_2^{\T{sa}}) \big) \in [0,\infty) .
\]

\begin{dfn}\label{d:admin}
We say that a semi-norm $L^{\T{sa}} \colon  V_1^{\T{sa}} \op V_2^{\T{sa}} \to [0,\infty)$ with $L^{\T{sa}}(1_{A_1} \op 1_{A_2}) = 0$ is \emph{admissible} when
\begin{enumerate}
\item $(V_1^{\T{sa}} \op V_2^{\T{sa}}, L^{\T{sa}})$ is an order unit compact quantum metric space;

\item $L_1^{\T{sa}} \colon  V_1^{\T{sa}} \to [0,\infty)$ and $L_2^{\T{sa}} \colon  V_2^{\T{sa}} \to [0,\infty)$ are the quotient semi-norms of $L^{\T{sa}}$ with respect to the surjections $\pi_1 \colon  V_1^{\T{sa}} \op V_2^{\T{sa}} \to V_1^{\T{sa}}$ and $\pi_2 \colon  V_1^{\T{sa}} \op V_2^{\T{sa}} \to V_2^{\T{sa}}$, respectively.
\end{enumerate}
\end{dfn}

The following definition is taken from \cite[Definition 4.2]{Rie:GHD}:

\begin{dfn}
The \emph{quantum Gromov-Hausdorff distance} between the order unit compact quantum metric spaces $(V_1^{\T{sa}},L_1^{\T{sa}})$ and $(V_2^{\T{sa}},L_2^{\T{sa}})$ is the quantity
\[
\T{dist}_q( V_1^{\T{sa}}, V_2^{\T{sa}}) := \inf\big\{ \T{dist}_H^{\rho_{L^{\T{sa}}}}\big( S(V_1^{\T{sa}}), S(V_2^{\T{sa}})\big) \mid L^{\T{sa}} \T{ is an admissible semi-norm} \big\} .
\]
When $(A_1,L_1)$ and $(A_2,L_2)$ are $C^*$-algebraic compact quantum metric spaces (as in Definition \ref{d:cqms}), we define their \emph{quantum Gromov-Hausdorff distance} by
\[
\T{dist}_q( (A_1,L_1), (A_2,L_2)) := \T{dist}_q( V_1^{\T{sa}}, V_2^{\T{sa}}).
\]
\end{dfn}
We denote, for $j=1,2$, the order unit norm on $V_j^{\T{sa}}$  by $\| \cd \|_j \colon  V_j^{\T{sa}} \to [0,\infty)$ and we denote the radius by $r_j \geq 0$. 
Following Li \cite{Li:GH-dist},  we introduce the notation
\begin{equation}\label{eq:ballinter}
\C D_R( V_j^{\T{sa}}) := \big\{ \xi \in V_j^{\T{sa}} \mid L_j^{\T{sa}}(\xi) \leq 1 \T{ and } \| \xi \|_j \leq R \big\}, \quad R \geq 0.
\end{equation}
We remark that these subsets are totally bounded in the topology coming from the norm $\| \cd \|_j$ as a consequence of Theorem \ref{t:quamet}. The following definition is taken from \cite[Definition 4.2]{Li:GH-dist}:

\begin{dfn}
For $R \geq r_1,r_2 \geq 0$, the \emph{$R$-order unit quantum Gromov-Hausdorff distance} between the order unit compact quantum metric spaces $(V_1^{\T{sa}},L_1^{\T{sa}})$ and $(V_2^{\T{sa}},L_2^{\T{sa}})$ is defined by
\[
\begin{split}
 & \T{dist}_{\T{oq}}^R(V_1^{\T{sa}},V_2^{\T{sa}})  \\
  & \q := \inf \Big\{ \max \big\{ \T{dist}_H^{\rho_{\| \cd \|}}\big( \io_1(\C D_R( V_1^{\T{sa}})), \io_2(\C D_R( V_2^{\T{sa}})) \big) 
, R \cd \| \io_1(1_{A_1}) - \io_2(1_{A_2}) \| \big\} \Big\},
\end{split}
\]
where the infimum is taken over all triples $(\io_1,\io_2,W)$, where $W$ is a normed real vector space and $\io_1 \colon  V_1^{\T{sa}} \to W$ and $\io_2 \colon  V_2^{\T{sa}} \to W$ are isometric $\rr$-linear maps. The notation $\rho_{\| \cd \|}$ refers to the metric on $W$ coming from the norm $\| \cd \|$ on $W$. 
\end{dfn}

We quote the following result from \cite[Proposition 4.8 and Proposition 4.10]{Li:GH-dist}:

\begin{prop}\label{p:lipghd}
For every $R \geq r_1,r_2 \geq 0$, we have the inequalities
\[
\frac{1}{2} \cd \T{dist}_{\T{oq}}^R(V_1^{\T{sa}},V_2^{\T{sa}}) \leq \T{dist}_q(V_1^{\T{sa}},V_2^{\T{sa}}) \leq 3 \cd \T{dist}_{\T{oq}}^R(V_1^{\T{sa}},V_2^{\T{sa}}) .
\]
\end{prop}
Thus, for families of compact quantum metric spaces with an upper bound $R$ on their radii, one may prove continuity in $\T{dist}_q$ by proving it in $\T{dist}_{\T{oq}}^R$ instead, and we will implement this strategy in Section \ref{s:unicri}.

\section{Semi-norms and non-isometric actions}\label{s:noniso}
In this section we aim to construct various semi-norms on algebras arising as crossed products with the integers. Later on, in Section \ref{s:comcro} we provide conditions ensuring that these semi-norms give rise to compact quantum metric spaces. For actions satisfying a suitable isometry criterion this was already studied in \cite{BMR:DSS} (see section \ref{ss:twist} for more details), but many naturally occurring actions do not fall into the class covered  in \cite{BMR:DSS}, and we will show below how to remedy this problem.
Throughout this section we consider a unital $C^*$-algebra $B$ equipped with a $*$-automorphism $\be \colon  B \to B$. Let $V_B \su B$ be a norm-dense complex subspace with $1 \in V_B$ and  assume that $\xi^* \in V_B$ for all $\xi \in V_B$. We fix a semi-norm $L_B \colon  V_B \to [0,\infty)$ with $L_B(1) = 0$ and with $L_B(\xi) = L_B(\xi^*)$ for all $\xi \in V_B$, and assume that the $*$-automorphism $\be \colon  B \to B$ and its inverse $\be^{-1} \colon  B \to B$ preserve the subspace $V_B$ so that
\[
\be(V_B) = V_B = \be^{-1}(V_B) .
\]
We emphasize that $\be \colon  V_B \to V_B$ is not assumed to be isometric with respect to the semi-norm $L_B \colon  V_B \to [0,\infty)$. In fact, we do not even assume that $\be \colon V_B \to V_B$ is bounded with respect to $L_B$. \\

Consider now the reduced crossed product $A := \cro{B}:=B\rtimes_{r} \ZZ$, which we represent on the Hilbert $C^*$-module $\ell^2(\zz) \hot B$ over $B$ via the left regular representation
\[
\la : \cro{B} \longrightarrow \B L\big( \ell^2(\zz) \hot B \big) ,
\] 
where $\B L\big( \ell^2(\zz) \hot B \big)$ denotes the unital $C^*$-algebra of bounded adjointable operators on the Hilbert $C^*$-module $\ell^2(\zz) \hot B$. We recall here that $\ell^2(\zz) \hot B$ agrees with the standard Hilbert $C^*$-module over the $C^*$-algebra $B$ as described in details in \cite[Chapter 1]{Lan:HCM}. The left regular representation is the injective $*$-homomorphism given by
\[
\la(b)(  e_m \ot c)  := e_m \ot \be^{-m}(b) \cd c \q \T{and} \q
\la(U)( e_m \ot c) := e_{m + 1} \ot c
\]
for all $b,c \in B$ and $m \in \zz$, where the sequence $\{ e_m \}_{m \in \zz}$ denotes the standard orthonormal basis of the Hilbert space $\ell^2(\zz)$. Notice that the relation $\la(U) \la(b) \la(U^*) = \la( \be(b))$ holds for all $b \in B$. We will often identify $\cro{B}$ with the image $\la(\cro{B}) \su \B L(\ell^2(\zz) \hot B)$. We are primarily interested in the reduced crossed product in contrast to the maximal crossed product since one of our main motivations is to understand the spectral geometry of crossed products in the context of spectral triples (equipped with a faithful representation).  

We now define the dense subspace
\[
V_A := C_c(\zz,V_B) \su A = \cro{B}
\]
as the smallest subspace of the reduced crossed product containing the vectors $\xi \cd U^n$ for all $\xi \in V_B$ and all $n \in \zz$. Remark that the unit $1 \in A$ belongs to $V_A$ and that $V_A$ is invariant under the involution on $A = \cro{B}$. Indeed, for $\xi \in V_B$ and $n \in \zz$ we have that
\[
(\xi \cd U^n)^* = \be^{-n}(\xi^*) \cd U^{-n} \in V_A
\]
by our standing assumptions. Denote by $\io \colon  B \to \cro{B}$ the unital $*$-homomorphism given by $\io(b) = \la(b)$. 
In order to introduce many interesting semi-norms on $V_A$ we make the following:

\begin{dfn}\label{d:vertiii}
We say that a norm $\vertiii{ \cd } : C_c(\zz) \to [0,\infty)$ is \emph{order preserving} when
\[
\bvert{ h } \leq \bvert{ k  },
\]
whenever $0 \leq h(n) \leq k(n)$ for all $n \in \zz$.
\end{dfn} 

For each $n \in \zz$ we let $e_n \in C_c(\zz)$ denote the compactly supported function which is $1$ at $n \in \zz$ and $0$ everywhere else.

For the remainder of this section we fix an order preserving norm $\vertiii{\cd } \colon C_c(\zz) \to [0,\infty)$, and, for convenience, we moreover assume that $\vertiii{\cd}$ is normalized such that $\vertiii{e_0} = 1$. 
We aim to study the two semi-norms $L_1$ and $L_2 \colon  V_A \to [0,\infty)$ defined by
\begin{equation}\label{eq:semi}
\begin{split}
& L_1( \eta ) := \big\| \sum_{n=-\infty}^{\infty} n \cd \eta(n) U^n \big\| \q \T{and} \\
& L_2( \eta ) := \max\big\{ \vertiii{ L_B \ci \eta} , \vertiii{ L_B \ci \eta^*} \big\} ,
\end{split}
\end{equation}
where the sum appearing is in fact finite and where the involution comes from the involution in $C_r^*(\zz,B)$. We remark that $L_j(\eta) = L_j(\eta^*)$ and $L_j(1) = 0$ for $j = 1,2$ and all $\eta \in V_A$. Let $\tau_0 \colon  \cro{B} \to B$ denote the conditional expectation given by $\tau_0( x ) := \inn{e_0 \ot 1_B, \la(x) (e_0 \ot 1_B)}$ for all $x \in \cro{B}$ and for each $n \in \zz$ we define $\tau_n \colon  \cro{B} \to B$ by $\tau_n(x) := \tau_0(x \cd U^n )$. Remark that for $\eta \in V_A$ and $n \in \zz$ we have that
\[
\eta(n) = \tau_{-n}(\eta) \q \T{and} \q \eta^*(n) = \be^n\big( \eta(-n)^*\big) = \tau_{-n}(\eta^*) .
\]
In particular, it holds that
\[
\begin{split}
L_2(\eta) & = \max\big\{  \bvert{ \sum_{n=-\infty}^\infty L_B(\tau_{-n}(\eta)) e_n} , \bvert{\sum_{n=-\infty}^\infty  L_B(\tau_{-n}(\eta^* )) e_n} \big\} \\
& = \max\big\{ \sup_{N \in \nn}\bvert{ \sum_{n = -N}^N L_B(\tau_{-n}(\eta)) e_n}, \sup_{N \in \nn}\bvert{ \sum_{n = -N}^N 
L_B(\tau_{-n}(\eta^* )) e_n} \big\}, 
\end{split}
\]
where the last equality follows since the norm $\vertiii{\cd}$ is assumed to be order preserving.

\begin{lemma}\label{l:lowsemi}
The semi-norm $L_2 \colon  V_A \to [0,\infty)$ is lower semi-continuous if and only if the semi-norm $L_B \colon  V_B \to [0,\infty)$ is lower semi-continuous.
\end{lemma}
\begin{proof}
Notice first that our normalization condition $\vertiii{e_0} = 1$ entails that $L_B = L_2 \ci \io \colon V_B \to [0,\infty)$ so that lower semi-continuity of $L_2 \colon  V_A \to [0,\infty)$ implies lower semi-continuity of $L_B \colon  V_B \to [0,\infty)$. Suppose now that $L_B \colon  V_B \to [0,\infty)$ is lower semi-continuous. For each $n \in \zz$, define the functions $g_n$ and $h_n \colon  V_A \to [0,\infty)$ by $g_n(\eta) := L_B( \tau_{-n}(\eta))$ and $h_n(\eta) := L_B( \tau_{-n}( \eta^*))$ for all $\eta \in V_A$. Then $g_n$ and $h_n$ are lower semi-continuous since $\tau_{-n}$ and $\tau_{-n} \ci * \colon  \cro{B} \to B$ are continuous and $L_B \colon  V_B \to [0,\infty)$ is lower semi-continuous by assumption. This shows that $L_2 \colon  V_A \to [0,\infty)$ is lower semi-continuous as well since 
\[
L_2(\eta) = \max\big\{ \sup_{N\in \nn} \bvert{ \sum_{n = -N}^N g_n(\eta) e_n }, \sup_{N\in \nn} \bvert{ \sum_{n = -N}^N h_n(\eta) e_n } \big\} 
\]
for all $\eta \in V_A$.
\end{proof}

For each $z \in \B T$, we have the unitary operator $V_z \colon  \ell^2(\zz) \hot B \to \ell^2(\zz) \hot B$ defined by $V_z( e_m \ot b) := z^m \cd e_m \ot b$ for all $b \in B$ and $m \in \zz$. This unitary operator induces a strongly continuous action of the circle $\al \colon  \B T \ti \cro{B} \to \cro{B}$ given by 
\begin{equation}\label{eq:dual}
\al_z( x) := V_z x V_z^* \q \T{for all } x \in \cro{B} \T{ and } z \in \B T,
\end{equation}
and which is referred to as the \emph{dual action}. We remark that 
\begin{align}\label{eq:dual-action}
\al_z( b \cd U^n) = z^n \cd b \cd U^n
\end{align}
 for all $b \in B$, $z \in \B T$ and $n \in \zz$.\\
The arc length on the compact group $\B T$ yields a length function $l \colon  \B T \to [0,\infty)$ given by
\begin{equation}\label{eq:length}
l( \exp(i t) ) := |t| \q \T{for all } t \in [-\pi,\pi].
\end{equation}
This length function in combination with the dual action give rise to a semi-norm $L_l \colon  V_A \to [0,\infty)$ defined by
\[
L_l(\eta) := \sup \left\{ \frac{\| \al_z(\eta) - \eta \|}{l(z)} \mathrel{\Big|} z \in \B T \sem \{1\} \right\}
\]
for all $\eta \in V_A$, see for example \cite{Rie:MSA}. Clearly, the semi-norm $L_l \colon  V_A \to [0,\infty)$ is lower semi-continuous. The next result is well-known, but for completeness we provide the short proof here:

\begin{lemma}\label{l:circle}
The two semi-norms $L_1$ and $L_l \colon  V_A \to [0,\infty)$ agree. In particular, $L_1 \colon  V_A \to [0,\infty)$ is lower semi-continuous.
\end{lemma}
\begin{proof}
Let $\eta = \sum_{n = -\infty}^\infty \xi_n U^n \in V_A$ be given. The map $\rr \to A$ given by $t \mapsto \al_{\exp(i t)}( \eta)$ is smooth and the derivative at $0 \in \rr$ is given by
\[
\pa(\eta) := \lim_{t \to 0}\frac{ \al_{\exp(i t)}(\eta) - \eta}{t} = i \cd \sum_{n = -\infty}^\infty n \xi_n U^n .
\]
It therefore holds that $L_1(\eta) = \| \sum_{n = -\infty}^\infty n \xi_n U^n \| = \| \pa(\eta) \|$ and we moreover obtain the estimate
\[
L_1(\eta) = \Big\| \lim_{t \to 0}\frac{ \al_{\exp(i t)}(\eta) - \eta}{t} \Big\| 
\leq \sup_{t \in \rr \sem \{0\}} \frac{\| \al_{\exp(i t)}(\eta) - \eta \|}{|t|} = L_l(\eta) .
\]
To obtain the reverse inequality, we note that the derivative of the map $t \mapsto \al_{\exp(it)}(\eta)$ at any $t_0 \in \rr$ is given by $\al_{\exp(i t_0)}( \pa(\eta))$. This implies that
\[
\frac{\| \al_{\exp(it)}(\eta) - \eta \|}{ |t|} = \frac{1}{|t|} \Big\| \int_0^t \al_{\exp(is)}( \pa(\eta)) ds \Big\|
\leq \| \pa(\eta) \| =L_1(\eta) 
\]
for all $t \in \rr \sem \{0\}$ and thus that $L_l(\eta) \leq L_1(\eta)$. 
\end{proof}

\subsection{Twisted derivations on crossed products}\label{ss:twist}
For the remainder of this section, we discuss some alternative densely defined semi-norms on $\cro{B}$ which are more tightly related to the noncommutative geometric structure of this reduced crossed product. These semi-norms generalize the semi-norms introduced in \cite{BMR:DSS} for equicontinuous actions in so far that our conditions on the $*$-automorphism $\be \colon  B \to B$ are considerably more relaxed. This increased degree of flexibility has the advantage of capturing examples occurring naturally in Riemannian geometry. We will elaborate on this in Example \ref{ex:diff}, Remark \ref{rem:diff} and Section \ref{sec:diff-of-riemannian}, but before doing so we treat the abstract theory in detail.  Some of the constructions appearing here are inspired by the modular techniques from \cite{Kaa:UKM}. \\ %

Suppose that $D \colon  \T{Dom}(D) \to H$ is an unbounded selfadjoint operator acting on the Hilbert space $H$ and that our unital $C^*$-algebra $B$ is represented in a faithful and unital way on $H$ via a $*$-homomorphism $\pi \colon  B \to \B L(H)$.
We assume that $\s B \su B$ is a dense unital $*$-subalgebra of $B$ such that $\pi(b)( \T{Dom}(D) ) \su \T{Dom}(D)$ and the commutator $[D, \pi(b)] \colon  \T{Dom}(D) \to H$ extends to a bounded operator $d(b) \colon  H \to H$ for all $b \in \s B$. In particular, we have the semi-norm
\[
L_D : \s B \to [0,\infty) \q L_D(b) := \| \ov{ [D,\pi(b)]} \| = \| d(b) \| .
\]
These conditions are satisfied when $(\s B,H,D)$ is a unital spectral triple, but we need not require that the resolvent $(i + D)^{-1} \colon  H \to H$ be compact. 
The next definition is a slightly modified version of \cite[Definition 2 \& 3]{BMR:DSS}, see also \cite[Lemma 2]{BMR:DSS}:
\begin{dfn}\label{d:quasi}
We say that $\be \in \T{Aut}(B)$ with $\be(\s B) = \s B$ is \emph{quasi-isometric} when there exist constants $\la, \mu \geq 1$ such that 
\[
\sup_{n \in \nn_0 }\mu^{-n} \cd \| d( \be^n(b)) \|_\infty < \infty \q \mbox{and} \q  
\sup_{n \in \nn_0 }\la^{-n} \cd \| d( \be^{-n}(b)) \|_\infty < \infty
\]
for all $b \in \s B$. We refer to such $\la,\mu \geq 1$ as \emph{quasi-isometry constants}. We say that $\be$ is \emph{equicontinuous} when
\[
\sup_{n \in \zz} \| d(\be^n(b)) \|_\infty < \infty \q \T{for all } b \in \s B .
\]
\end{dfn}

We remark that the above definition could be extended to the more general context of semi-norms that do not necessarily arise from a spectral triple. Our primary example of a quasi-isometric action here below is contained within the framework of spectral triples and we shall therefore not comment any further on the more general context. 

\begin{example}\label{ex:diff}
In this example we show how quasi-isometric actions arise naturally in differential geometry. To this end, let $M$ be a compact Riemannian manifold, $E \to M$ be a complex hermitian vector bundle and $\s D \colon \Ga^\infty(M,E) \to L^2(M,E)$ be a first order, symmetric and elliptic differential operator. Let moreover $\psi \colon M \to M$ be a diffeomorphism. With $D \colon \T{Dom}(D) \to L^2(M,E)$ denoting the closure of $\s D$ we obtain a spectral triple $(C^\infty(M), L^2(M,E), D)$ and the corresponding semi-norm $L_D \colon C^\infty(M) \to [0,\infty)$ can be computed in terms of the principal symbol $\si_D \colon T^* M \to \T{End}(E)$, see for example \cite[Chapter 10]{HiRo:AKH}. Indeed, suppressing the action of both $C(M)$ and the action of the continuous sections of $\T{End}(E)$ on $L^2(M,E)$ we have the relation
\[
\ov{ [D,f] } = \si_D( df )  \q \T{for all } f \in C^\infty(M),
\]
where $d \colon C^\infty(M) \to \Ga^\infty(M,T^* M)$ denotes the exterior derivative. Associated to $\psi$ and $\psi^{-1}$ we have the exterior derivatives $d \psi \colon TM \to \psi^* TM$ and $d \psi^{-1} \colon TM \to (\psi^{-1})^* TM$; in particular, we have the operator norms
\[
\| d \psi \| := \sup_{p \in M} \| d \psi(p) \|_{\B L(T_p M, T_{\psi(p)}M)} \q 
\| d \psi^{-1} \| := \sup_{p \in M} \| d \psi^{-1}(p) \|_{\B L(T_p M, T_{\psi^{-1}(p)}M)} \\
\]
It can then be verified that the $*$-automorphism $\be \colon C(M) \to C(M)$ defined by $\be(f) := f \ci \psi$ is quasi-isometric using the quasi-isometry constants
\[
\la := \max\{ 1, \| d \psi^{-1} \| \} \q \T{and} \q \mu := \max\{ 1, \| d \psi \| \} .
\]
\end{example}
\medskip

We now return to the general setting and fix, for the rest of this section, a $\be \in \T{Aut}(B)$ with $\be(\s B) = \s B$. We may lift the unbounded selfadjoint operator $D$ to an unbounded selfadjoint operator $\T{diag}(D) \colon  \T{Dom}( \T{diag}(D)) \to \ell^2(\zz) \hot H$ defined by
\[
\T{diag}(D)\Big( \sum_{m = -\infty}^\infty e_m \ot \xi_m\Big) := \sum_{m = -\infty}^\infty e_m \ot D(\xi_m)
\]
for all vectors in the domain $\T{Dom}(\T{diag}(D))$, which is defined by
\[
\T{Dom}(\T{diag}(D)) := \Big\{ \sum_{m = -\infty}^\infty e_m \ot \xi_m \in \ell^2(\zz) \hot H \mid \xi_m \in \T{Dom}(D), \sum_{m = -\infty}^\infty \| D(\xi_m)\|^2<\infty \Big\} .
\]
We remark that the compactly supported sequences in $\T{Dom}(D)$, i.e.~the subspace $C_c(\zz, \T{Dom}(D)) \su \ell^2(\zz) \hot H$, form a core for $\T{diag}(D)$. 
We may also use the faithful unital representation $\pi \colon  B \to \B L(H)$ to obtain a left regular representation of the reduced crossed product $\la_\pi \colon  \cro{B} \to \B L( \ell^2(\zz) \hot H)$ given by 
\[
\la_\pi(b)(e_m \ot \xi) := e_m \ot \pi(\be^{-m}(b))(\xi) \q \T{and} \q \la_\pi(U)(e_m \ot \xi) := e_{m + 1} \ot \xi .
\]
We sometimes suppress the left regular representation $\la_\pi$ from the notation. It would be natural to use the commutator with $\T{diag}(D)$ and the representation $\la_\pi$ to define a derivation on the unital $*$-algebra $\s A := C_c(\zz,\s B)$. However a simple computation reveals that
\[
[ \T{diag}(D), \la_\pi(b) ](e_m \ot \xi) = e_m \ot d( \be^{-m}(b))(\xi)
\]
for all $m \in \zz$, $\xi \in \T{Dom}(D)$ and $b \in \s B$. In particular, the commutator $[ \T{diag}(D), \la_\pi(b) ]$ extends to a bounded operator if and only if the $*$-automorphism $\be \colon  B \to B$ is equicontinuous, in the sense of Definition \ref{d:quasi}.  Let us now instead only assume that the $*$-automorphism $\be \colon  B \to B$ is quasi-isometric with respect to the derivation $d \colon  \s B \to \B L(H)$ and fix the quasi-isometry constants $\la,\mu \geq 1$ as in Definition \ref{d:quasi}. We may then control the unbounded behaviour of the commutator $[\T{diag}(D), \la_\pi(b)] \colon  C_c(\zz,\T{Dom}(D)) \to \ell^2(\zz) \hot H$ by introducing a modular operator: 

\begin{dfn}
In the setting above we define the \emph{modular operator} $\Ga \colon  \ell^2(\zz) \hot H \to \ell^2(\zz) \hot H$ by the formula $\Ga( e_m \ot \xi) := e_m \ot \ka(m) \cd \xi$, where the constants $\ka(m) \in (0,1]$ are given by
\[
\ka(m) := \fork{ccc}{\la^{-m} & \T{for} & m \geq 0 \\ \mu^m & \T{for} & m < 0} .
\]  
\end{dfn}

Since $\la,\mu \geq 1$, the modular operator is bounded and strictly positive and the inverse $\Ga^{-1} \colon  \T{Im}(\Ga) \to \ell^2(\zz) \hot H$ has the explicit description $\Ga^{-1}( e_m \ot \xi) := e_m \ot \ka(m)^{-1} \xi$ on the invariant core $C_c(\zz,H) \subseteq \T{Im}(\Ga) = \T{Dom}(\Ga^{-1})$. When $\la$ or $\mu$ is different from $1$, we have that $\Ga^{-1}$ is an unbounded positive and selfadjoint operator. 
We apply the modular operator $\Ga$ to ``rescale'' the unbounded selfadjoint operator $\T{diag}(D)$ and consider the \emph{modular lift} $D_{\Ga} := \T{diag}(D) \Ga \colon  \T{Dom}(D_\Ga) \to \ell^2(\zz) \hot H$ defined by
\[
D_{\Ga}\Big( \sum_{m = -\infty}^\infty e_m \ot \xi_m\Big) := \sum_{m = -\infty}^\infty e_m \ot \ka(m) \cd D(\xi_m),
\]
on the domain given by
\[
\T{Dom}(D_\Ga) := \Big\{ \sum_{m = -\infty}^\infty e_m \ot \xi_m \mid \xi_m\in \T{Dom}(D), \sum_{m = -\infty}^\infty e_m \ot \ka(m) \cd D(\xi_m)
\in \ell^2(\zz) \hot H \Big\},
\]
see \cite[Section 5]{Kaa:UKM}. 
The modular lift $D_\Ga$ is again an unbounded selfadjoint operator on $\ell^2(\zz) \hot H$ and it has $C_c(\zz,\T{Dom}(D)) \su \ell^2(\zz) \hot H$ as a core. It turns out that the modular lift also does not have a well-behaved commutator with the elements in the unital $*$-algebra $\s A=C_c(\ZZ, \s B)$, but that one may instead twist the commutator by a couple of \emph{modular representations} and in this way obtain a twisted derivation $\pa_\Ga \colon  \s A \to \B L( \ell^2(\zz) \hot H)$. We now explain how this works.

\begin{lemma}\label{l:modular}
Let $x \in \s A = C_c(\zz,\s B)$. Then the densely defined operator
\[
\Ga^{1/2} \la_\pi(x) \Ga^{-1/2} \colon  \T{Im}(\Ga^{1/2}) \to \ell^2(\zz) \hot H 
\]
extends to a bounded operator, which we denote by $\si^{1/2}(x) \colon  \ell^2(\zz) \hot H \to \ell^2(\zz) \hot H$.
%
\end{lemma}
\begin{proof}
It suffices to consider the cases where $x = U$, $x = U^*$ or $x = b$ for some $b \in \s B$. The case where $x = b$ is a triviality and here we obtain that $\si^{1/2}(b) = \la_\pi(b)$. So we focus on the situation where $x = U$ or $x = U^*$. Let $P \colon  \ell^2(\zz) \hot H \to \ell^2(\zz) \hot H$ denote the orthogonal projection with image $\ell^2(\nn_0 ) \hot H \su \ell^2(\zz) \hot H$. For each $\xi \in H$ and $m \in \B Z$, we compute that
\[
\begin{split}
 \Ga^{1/2} \la_\pi(U) \Ga^{-1/2}(e_m \ot \xi) 
&  = \Ga^{1/2} \la_\pi(U) \la^{m/2} P (e_m \ot \xi)
+ \Ga^{1/2} \la_\pi(U) \mu^{-m/2} (1 - P) (e_m \ot \xi) \\
&  = \la^{-1/2} \cd \la_\pi(U) P (e_m \ot \xi)
+ \mu^{1/2} \cd \la_\pi(U) (1 - P) (e_m \ot \xi) .
\end{split}
\]
Since the subspace $C_c(\B Z,H)$ is a core for $\Ga^{-1/2}$ this implies that $\si^{1/2}(U) = \la^{-1/2} \la_\pi(U) P + \mu^{1/2} \la_\pi(U) (1 - P)$. A similar computation shows that $\si^{1/2}(U^*) = \la^{1/2} P \la_\pi(U^*) + \mu^{-1/2}(1 - P) \la_\pi(U^*)$. This proves the lemma.
\end{proof}

We define the algebra homomorphism $\si^{1/2} \colon  \s A \to \B L(\ell^2(\zz) \hot H)$ using the above lemma, and the algebra homomorphism $\si^{-1/2} \colon  \s A \to \B L( H \hot \ell^2(\zz))$ by putting $\si^{-1/2}(x) := \si^{1/2}(x^*)^*$ for all $x \in \s A$. We refer to these two algebra representations as the \emph{modular representations}.

\begin{lemma}\label{lem:commutator-formulas}
Suppose that the $*$-automorphism $\be \colon  B \to B$ is quasi-isometric and let $x \in \s A = C_c(\zz,\s B)$. Then $\si^{-1/2}(x)$ preserves the core $C_c(\zz,\T{Dom}(D))$ for the modular lift $D_\Ga$, and the twisted commutator
\[
D_\Ga \si^{-1/2}(x) - \si^{1/2}(x) D_\Ga \colon  C_c(\zz,\T{Dom}(D)) \to \ell^2(\zz) \hot H
\]
extends to a bounded operator, which we denote by $\pa_\Ga(x) \colon  \ell^2(\zz) \hot H \to \ell^2(\zz) \hot H$.
\end{lemma}
\begin{proof}
We may assume that $x = b U^n$ for some $b \in \s B$ and $n \in \zz$. It is clear that $\si^{-1/2}(x)$ preserves the core $C_c(\zz,\T{Dom}(D))$, indeed for each $\eta \in C_c(\zz,\T{Dom}(D))$ we have that $\si^{-1/2}(x)(\eta) = \Ga^{-1/2} x \Ga^{1/2}(\eta)$. On the core, we moreover compute that
\[
\begin{split}
\big( D_\Ga \si^{-1/2}(x) - \si^{1/2}(x) D_\Ga \big)(\eta) 
& = \Ga^{1/2} [ \T{diag}(D), \la_\pi(x) ] \Ga^{1/2} (\eta) \\ 
& = \Ga^{1/2} [ \T{diag}(D), \la_\pi(b) ] \Ga^{1/2} \si^{-1/2}(U^n)(\eta) .
\end{split}
\]
It thus suffices to show that the diagonal operator $\Ga^{1/2} [\T{diag}(D),\la_\pi(b)] \Ga^{1/2} \colon  C_c(\zz,\T{Dom}(D)) \to \ell^2(\zz) \hot H$ extends to a bounded operator on $\ell^2(\zz) \hot H$. Now, for each $\xi \in \T{Dom}(D)$ and $m \in \zz$, we have that
\[
\Ga^{1/2} [\T{diag}(D),\la_\pi(b)] \Ga^{1/2}( e_m \ot \xi) = e_m \ot \ka(m) d( \be^{-m}(b))(\xi) ,
\]
and the assumption that $\be \colon  B \to B$ is quasi-isometric (with respect to the quasi-isometry constants $\la,\mu \geq 1$) implies that $\sup_{m \in \zz} \| \ka(m) d(\be^{-m}(b)) \|_\infty < \infty$. This indeed shows that the diagonal operator 
\[
\Ga^{1/2} [\T{diag}(D),\la_\pi(b)] \Ga^{1/2} \colon  C_c(\zz,\T{Dom}(D)) \to \ell^2(\zz) \hot H
\]
extends to a bounded operator and the lemma is proved.
\end{proof}

The linear map $\pa_\Ga \colon  \s A \to \B L(\ell^2(\zz) \hot H)$ is a \emph{twisted} $*$-derivation meaning that
\begin{equation}\label{eq:starder}
\pa_\Ga(xy) = \si^{1/2}(x) \pa_\Ga(y) + \pa_\Ga(x) \si^{-1/2}(y) \q \T{and} \q
\pa_\Ga(x^*) = -  \pa_\Ga(x)^*  
\end{equation}
for all $x,y \in \s A$. 
As in the proof of Lemma \ref{l:circle}, we also have the $*$-derivation $ -i \pa \colon  \s A \to \s A$ defined by $ -i \pa(x)(n) = n \cd x(n)$, which is tightly related to the dual action on the reduced crossed product $\cro{B}$. Combining these two derivations we obtain an interesting semi-norm $L_\Ga \colon  \s A \to [0,\infty)$ defined by
\begin{equation}\label{eq:semungra}
L_\Ga(x) := \max\{ \| \pa_\Ga(x) + \pa(x) \|_\infty, \| \pa_\Ga(x) - \pa(x) \|_\infty\}
\end{equation}
for all $x \in \s A$. \\
In the case where $(\s B,H,D)$ is a unital spectral triple and the $*$-automorphism $\be \colon  B \to B$ is equicontinuous, we may choose $\la = \mu = 1$ so that $\Ga = \T{id} \colon  \ell^2(\zz) \hot H \to \ell^2(\zz) \hot H$. The semi-norm $L_\Ga = L_{\T{id}}$ then coincides exactly with the semi-norm investigated in \cite[Section 3]{BMR:DSS} and can be seen to arise from a unital spectral triple over the reduced crossed product $\cro{B}$, see also \cite{HSWZ:STC,Pat:CST}. In the more general setting where the $*$-automorphism $\be \colon  B \to B$ is only quasi-isometric we still have the semi-norm $L_\Ga \colon  \s A \to [0,\infty)$ and we shall relate it to metrics on the state space of $\cro{B}$ in Section \ref{s:comcro}. In the general context, we note that there is no reason to expect that the semi-norm $L_\Ga$ comes from a (twisted) spectral triple, since the operation $\pa_\Ga$ is a twisted derivation whereas the operation $\pa$ is a (non-twisted) derivation. \\

We end this section by remarking that, if the Hilbert space $H$ carries a $\zz/2\zz$-grading operator $\ga \colon  H \to H$ such that $\pi(b)$ is even for all $b \in B$ and $D \colon  \T{Dom}(D) \to H$ is odd, then the semi-norm $L_\Ga$ should be modified accordingly and be replaced by
\begin{equation}\label{eq:semgra}
L_\Ga(x) := \| \pa_\Ga(x) -i (1 \ot \ga)  \pa(x) \|_\infty
\end{equation}
for all $x \in \s A$. This is due to the different structure of the (unbounded) Kasparov product in the cases $KK_1(\cro{B},B) \ti K^1(B) \to K^0(\cro{B})$ and $KK_1(\cro{B},B) \ti K^0(B) \to K^1(\cro{B})$, which is indeed serving as a guideline for the constructions carried out in the present section. 

\section{Compact quantum metric space structures on crossed products}\label{s:comcro}
We now return to the general setting described in the beginning of  Section \ref{s:noniso}. We thus consider a unital $C^*$-algebra $B$ with  a fixed norm-dense, unital $*$-invariant subspace $V_B\su B$, a $*$-automorphism $\beta\colon B\to B$ with $\beta(V_B)=V_B$ and a $*$-invariant semi-norm $L_B\colon V_B \to [0,\infty)$ satisfying that $L_B(1)=0$. We moreover fix an order preserving norm $\vertiii{\cd} : C_c(\zz) \to [0,\infty)$ which we normalize so that $\vertiii{e_0} = 1$; see Definition \ref{d:vertiii}.
We are going to investigate the semi-norm 
\[
L_A := \max\{L_1,L_2\} \colon  V_A \to [0,\infty),
\]
where we recall that $V_A = C_c(\zz,V_B)$ is a norm-dense complex vector subspace of the reduced crossed product $A = \cro{B}$. The individual semi-norms $L_1$ and $L_2 \colon  V_A \to [0,\infty)$ are given by the explicit formulae in \eqref{eq:semi}. Our first main result is the following:

\begin{thm}\label{t:quacross}
Suppose that the densely defined semi-norm $L_B \colon  V_B \to [0,\infty)$ is lower semi-continuous and that $(B,L_B)$ is a compact quantum metric space (in the sense of Definition \ref{d:cqms}). Suppose moreover that the $*$-automorphism $\be \colon  B \to B$ satisfies that $\be(V_B) = V_B$. Then the semi-norm $L_A \colon  V_A \to [0,\infty)$ is lower semi-continuous and $(\cro{B},L_A)$ is a compact quantum metric space. Moreover, letting $r_B$ and $r_A \geq 0$ denote the radii of $(B,L_B)$ and $(\cro{B},L_A)$ we have the estimate $r_A \leq r_B + \frac{\pi}{2}$. 
\end{thm}

The proof of Theorem \ref{t:quacross} relies on a general theorem due to Li, \cite[Theorem 4.1]{Li:DCQ}, which in turn is related to methods developed by Rieffel in \cite[Theorem 2.3]{Rie:MSA}. Before presenting our proof we need a bit more notation: For each $n \in \zz$, we have the character $\varphi_n \in \hat{\B T}$ given by $\varphi_n(z) := z^n$ for all $z \in \B T$. For any finite linear combination of characters $\varphi = \sum_{n = -\infty}^\infty \la_n \varphi_n$, we define the linear map $\al_\varphi \colon  \cro{B} \to \cro{B}$ by
\[
\al_{\varphi}(x) := \frac{1}{2 \pi} \int_0^{2 \pi} \varphi( \exp(it)) \al_{\exp(it)}(x) dt 
\]
for all $x \in \cro{B}$, where $\al \colon  \B T \ti \cro{B} \to \cro{B}$ denotes the dual action introduced in \eqref{eq:dual}. By \eqref{eq:dual-action},  for each $x = \sum_{n = -\infty}^\infty b_n U^n \in C_c(\B Z,B)$ it holds that
\[
\al_{\varphi}( x ) = \sum_{n = -\infty}^\infty \la_{-n} \cd b_n U^n 
\]
and thus that $\al_{\varphi}(V_A) \su V_A$. For each $n \in \zz$, we let $A_n \su A$ denote the spectral subspace associated to the action $\al \colon  \B T \ti \cro{B} \to \cro{B}$. Explicitly, we have that
\[
A_n = \big\{ x \in A \mid \al_z(x) = z^n \cd x \T{ for all } z \in \B T \big\}
= \{ b U^n \mid b \in B \} = \T{Im}( \al_{\varphi_{-n}}),
\]
where the middle equation follows from \cite[Proposition 7.8.9]{Ped:CAA}.
As a Banach space, the spectral subspace $A_n \su A$ is thus isometrically isomorphic to $B$ via the linear isometry $\io_n \colon  B \to A_n$ defined by $\io_n(b) := b U^n$. We remark that the inverse of $\io_n \colon  B \to A_n$ is the map $\tau_{-n}|_{A_n} \colon  A_n \to B$, which we introduced before Lemma \ref{l:lowsemi}.

\begin{proof}[Proof of Theorem \ref{t:quacross}]
Lower semi-continuity of the semi-norm $L_A = \max\{L_1,L_2\} \colon  V_A \to [0,\infty)$ follows immediately from Lemma \ref{l:lowsemi} and Lemma \ref{l:circle}. To prove the remainder of the theorem we apply \cite[Theorem 4.1]{Li:DCQ} and need to verify the conditions of that theorem. This suffices to prove the present theorem. The relevant conditions are as follows:
\begin{enumerate}
\item there exists a constant $C > 0$ such that $L_l(\eta) \leq C \cd L_A(\eta)$ for all $\eta \in V_A$;
\item for any finite linear combination of characters $\varphi$ it holds that
\[
L_A( \al_\varphi(\eta)) \leq \frac{1}{2\pi}\int_0^{2\pi} | \varphi(\exp(it)) | dt \cd L_A(\eta) =\|\varphi\|_1\cdot L_A(\eta) 
\]
for all $\eta \in V_A$;
\item for each $n \in \B Z \sem \{0\}$, the subset $\big\{ x \in A_n \cap V_A \mid L_A(x) \leq 1 \big\} \su A_n$ is totally bounded;
\item the pair $(A_0, L_A|_{V_A \cap A_0})$ is a compact quantum metric space.
\end{enumerate}
Condition $(1)$, with $C = 1$, follows immediately from Lemma \ref{l:circle} and the definition of $L_A$. To prove condition $(2)$ we let $\varphi = \sum_{n = -\infty}^\infty \la_n \varphi_n$ be a finite linear combination of characters. We know from Lemma \ref{l:lowsemi} and Lemma \ref{l:circle} that $L_A = \max\{L_1,L_2\} \colon  V_A \to [0,\infty)$ is lower semi-continuous and we have remarked earlier that $\al_{\varphi}(V_A) \su V_A$. Recall next that $\al_z \colon  \cro{B} \to \cro{B}$ is an isometry for all $z \in \B T$  from which it follows that $L_1$ is $\alpha$-invariant, and $L_2$ is also seen to be $\alpha$-invariant since $|z|=1$, so it follows that $L_A(\al_z(\eta)) = L_A(\eta)$ for all $z \in \B T$ and $\eta \in V_A$. The relevant estimate is now a consequence of \cite[Lemma 4.3]{Li:DCQ}. To prove condition $(3)$ we let $n \in \B Z \sem \{0\}$ be given and notice that 
\[
L_A(\xi U^n) = \max\left\{ |n| \cd \| \xi \|_B , L_B(\xi) \vertiii{e_n},  L_B( \be^{-n}(\xi^*) ) \vertiii{e_{-n}} \right\} 
\]
for all $\xi \in V_B$. This observation implies that we have the inclusion
\[
\big\{ x \in A_n \cap V_A \mid L_A(x) \leq 1 \big\} \su \io_n\big( \B B_{ 1/\vertiii{e_n} }(L_B) \cap \B B_{1/|n|}( \| \cd \|_B ) \big) ,
\]
where we use the notation for balls introduced in \eqref{eq:balls}. Since $(B,L_B)$ is a compact quantum metric space by assumption we know that the intersection of balls $\B B_{1/\vertiii{e_n}}(L_B) \cap \B B_{1/|n|}( \| \cd \|_B ) \su B$ is totally bounded so the subset $\big\{ x \in A_n \cap V_A \mid L_A(x) \leq 1 \big\} \su A_n$ must be totally bounded as well. Condition $(4)$ is another way of saying that $( B, L_B)$ is a compact quantum metric space.  In fact, our normalization condition $\vertiii{e_0} = 1$ implies that $L_B = L_A \ci \io \colon V_B \to [0,\infty)$. This shows that $(\cro{B},L_A)$ is a compact quantum metric space. The upper bound on the radius $r_A \geq 0$ is also a consequence of \cite[Theorem 4.1]{Li:DCQ}, indeed we obtain that 
\[
r_A \leq r_B + \frac{1}{2\pi} \int_0^{2\pi} l( \exp(it)) dt = r_B + \frac{\pi}{2},
\]
where the length function $l \colon  \B T \to [0,\infty)$ is given explicitly in \eqref{eq:length}. This ends the proof of the theorem.
\end{proof}

We now apply the above Theorem \ref{t:quacross} to the semi-norm coming from the derivations described in Subsection \ref{ss:twist}; in particular we generalize a result of Bellissard, Marcolli and Reihani, see \cite[Proposition 3]{BMR:DSS}. Indeed we here consider quasi-isometric actions instead of the more restrictive equicontinuous actions corresponding to the case where the modular operator $\Ga$ equals the identity operator. Notice that the modular operator $\Ga \in \B L(\ell^2(\zz) \hot H)$ depends on a choice of quasi-isometry constants $\la,\mu \geq 1$ even though this is not reflected in the notation.

%

\begin{thm}\label{cor:BMR-cor}
Suppose that $(\s B,H,D)$ is a spectral triple, where the associated $*$-homomorphism $\pi \colon  B \to \B L(H)$ is unital and faithful, and suppose that the semi-norm $L_B \colon  \s B \to [0,\infty)$ given by $L_B(b) := \| d(b)\|_\infty$ turns $(B,L_B)$ into a compact quantum metric space. Suppose moreover that the $*$-automorphism $\be \colon  B \to B$ is quasi-isometric (in the sense of Definition \ref{d:quasi}). Then for any choice of quasi-isometry constants $\la,\mu \geq 1$ it holds that $(\cro{B},L_\Ga)$ is a compact quantum metric space, where the semi-norm $L_{\Ga} \colon  C_c(\zz,\s B) \to [0,\infty)$ is given by \eqref{eq:semungra} in the ungraded case and by \eqref{eq:semgra} in the graded case.  
\end{thm}
\begin{proof}
We remark that $L_B \colon  \s B \to [0,\infty)$ is automatically lower semi-continuous, see \cite[Proposition 3.7]{Rie:MSS}. Let the quasi-isometry constants $\la,\mu \geq 1$ be given. By Theorem \ref{t:quamet} we just need to check that $q( \B B_1(L_{\Ga})) \su \cro{B} / \cc 1$ is totally bounded. To this end, we define an order preserving norm $\vertiii{\cd} : C_c(\zz) \to [0,\infty)$ by the formula
\[
\vertiii{ h } := \sup_{n \in \zz} c^{|n|} | h(n) | , 
\]
where $c := \min\{ \la^{-1/2},\mu^{-1/2} \}$. Notice also that $\vertiii{e_0} = 1$, meaning that $\vertiii{\cdot}$ is normalized as required in Section \ref{s:noniso}. We let $L_A \colon \s A \to [0,\infty)$ denote the corresponding semi-norm so that $(A,L_A)$ is a compact quantum metric space by Theorem \ref{t:quacross}. It therefore suffices to show that $L_A(x) \leq L_{\Ga}(x)$ for all $x \in \s A = C_c(\zz,\s B)$, see Theorem \ref{t:quamet}.  Let $x \in \s A$ be given. An easy application of the triangle inequality shows that $\|\pa_{\Ga}(x)\|_\infty, \|\pa(x)\|_\infty\leq L_{\Ga}(x)$, and we moreover have that $L_1(x) = \| \pa(x) \|_\infty$, so it suffices to show that $L_2(x) \leq \| \pa_{\Ga}(x) \|_\infty$.  We now claim that 
\begin{equation}\label{eq:estimate}
c^{|n|} L_B( x(n)) = c^{|n|} \| d( x(n) ) \|_\infty \leq \| \pa_{\Ga}(x) \|_\infty \q \T{for all } n \in \zz .
\end{equation}
Establishing this claim proves the theorem; indeed, since $\s A$ is a $*$-algebra this will imply that 
\[
\begin{split}
L_2(x) & = \max\big\{  \bvert{ L_B \ci x} , \bvert{ L_B \ci x^*} \big\} \\
& = \max\big\{ \sup_{n \in \zz }c^{|n|} \| d(x(n)) \|_\infty , \sup_{n \in \zz } c^{|n|} \| d(x^*(n)) \|_\infty  \big\} \\
& \leq \max\big\{ \| \pa_\Ga(x) \|_\infty , \| \pa_\Ga(x^*) \|_\infty \big\} = \| \pa_\Ga(x) \|_\infty ,
\end{split}
\]
where the last identity follows from \eqref{eq:starder}. Let thus $n \in \zz$ be given. We define $g := \si^{-1/2}(U)$ and notice that $g^{-1} = \si^{-1/2}(U^*)$. It follows from the explicit formulae in the proof of Lemma \ref{l:modular} that
\begin{alignat*}{3}
g & = \la^{1/2} U P + \mu^{-1/2} U (1 - P) \q  &&\T{and} \q g^{-1} &&= \la^{-1/2} P U^* + \mu^{1/2} (1 - P) U^* \\
g^* & = \la^{1/2} P U^* + \mu^{-1/2}(1 - P) U^* \q  && \T{and} \q ( g^{-1} )^* &&= \la^{-1/2} U P  + \mu^{1/2} U (1 - P) ,
\end{alignat*}
where $P\in \B{L} (\ell^2(\ZZ) \hot H)$ denotes the orthogonal projection onto $\ell^2(\nn_0) \hot H$. 
In particular, we have that
\[
c^{2 |n| } = \min\{ \la^{-|n|},\mu^{-|n|} \} \leq g^{|n|}\left(g^{|n|}\right)^*, g^{-|n|} (g^{-|n|})^* . 
\]

Secondly, we see that $g  = \la^{1/2} U P + \mu^{-1/2} U (1 - P)$ has the form $D_1U$ for a bounded diagonal operator $D_1$ and since conjugating $D_1$ with $U$ results in another diagonal operator we conclude that there exists some bounded diagonal operator $D_n$ such that $g^n=D_nU^n$ for $n\geq 0$, and the same reasoning holds when $n<0$. 
Hence
\begin{align*}
\pa_\Gamma(x)) &=\sum_{n=-\infty}^\infty\ov{ \Ga^{1/2} \big[\T{diag}(D),\la_\pi(x(n))\big] \Ga^{1/2}} \cd g^n \\
& =\sum_{n=-\infty}^\infty\ov{ \Ga^{1/2} \big[\T{diag}(D),\la_\pi(x(n))\big] \Ga^{1/2}} \cd D_n U^n,  
\end{align*}
where the operator $\ov{ \Ga^{1/2} [\T{diag}(D),x(n)] \Ga^{1/2}} \cd D_n$ is diagonal.
We may thus estimate as follows
\[
\begin{split}
& c^{2 |n|} \| d( x(n) ) \|_\infty^2 \leq c^{2 |n|}  \sup_{k\in \zz} \big\| \kappa(k) d\big(\beta^{-k}(x(n))\big)  \big\|_\infty^2  
=c^{2 |n|} \big\| \ov{ \Ga^{1/2} \big[\T{diag}(D),\la_\pi(x(n))\big] \Ga^{1/2}} \big\|_\infty^2 \\
& \q \leq \big\| \ov{ \Ga^{1/2} \big[\T{diag}(D),\la_\pi(x(n))\big] \Ga^{1/2}} \cd g^n \big\|_\infty^2 
= \big\| \ov{ \Ga^{1/2} \big[\T{diag}(D),\la_\pi(x(n))\big] \Ga^{1/2}} \cd D_n \big\|_\infty^2 \\
& \q = \| \tau_0( \pa_\Ga(x) U^{-n}) \|_\infty^2 \leq \|\pa_\Ga( x ) \|_\infty^2 ,
\end{split}
\]
where $\tau_0 : \B L(\ell^2(\zz) \hot H) \to \B L(\ell^2(\zz) \hot H)$ denotes the contraction which puts all off-diagonal elements equal to zero. This proves \eqref{eq:estimate} and hence the result.
\end{proof}

\begin{remark}\label{rem:diff}
In this remark we return to the context of Example \ref{ex:diff}. The ellipticity assumption on $\s D \colon \Ga^\infty(M,E) \to L^2(M,E)$ implies that the semi-norm $L_D \colon C^\infty(M) \to [0,\infty)$ is equivalent to the semi-norm $L_{C(M)} \colon C^\infty(M) \to [0,\infty)$ defined by
\[
L_{C(M)}(f) := \sup_{p \in M}\| df(p) \|_{T^*_pM} .
\]
If we furthermore assume that the compact Riemannian manifold $M$ is connected we shall see in Proposition \ref{p:quametrie} that the pair $(C(M),L_{C(M)})$ is a $C^*$-algebraic compact quantum metric space. Using the equivalence of semi-norms this entails that $(C(M), L_D)$ is a $C^*$-algebraic compact quantum metric space as well. We thereby obtain a large class of examples satisfying the assumptions of Theorem \ref{cor:BMR-cor}.
\end{remark}



\section{A uniform criterion for quantum metric convergence}\label{s:unicri}

In this section we study a general uniform criterion for quantum Gromov-Hausdorff metric convergence of a family of order unit compact quantum metric spaces sitting as real vector subspaces of the fibres of a continuous field of unital $C^*$-algebras. This uniform criterion will then be applied to crossed products in Section \ref{sec:qgh-convergence-for-crossed-products}. Our criterion is related to, but different from, Li's uniform criterion introduced in \cite[Definition 7.9]{Li:GH-dist} with the aim of proving  quantum Gromov-Hausdorff metric convergence of $\theta$-deformations. Indeed, we are mainly controlling the variation of the semi-norms in a uniform fashion whereas the $C^*$-norms are only required to be uniformly equivalent to a  fixed reference norm, so that our setup in that respect has some similarity with \cite[Section 10]{Rie:GHD}. The motivation for introducing this hybrid of Li's and Rieffel's notions stems from our main example, studied in detail in Section \ref{sec:diff-of-riemannian}, consisting of crossed products arising from a  family of diffeomorphisms of a Riemannian manifold, where we could not see how to apply Li's uniform criterion, but, as the reader shall see, we can indeed single out a reference norm. In fact, such a reference norm exists in the full generality for fields of crossed products by the integers, at least when working with an upper bound on the number of Fourier coefficients;  see e.g.~Section \ref{sec:qgh-convergence-for-crossed-products} for more on this.\\

Throughout this section, we let $\{ A_t\}_{t \in T}$ be a continuous field of unital $C^*$-algebras over a compact Hausdorff space $T$ arising from an injective unital $*$-homomorphism $j \colon  C(T) \to \C Z(A)$, where $A$ is a unital $C^*$-algebra, see \cite{Rie:CFG} and \cite{Bla:DCH} for details.
We denote the norm, unit, and involution in $A_t$ by $\|\cdot\|_t$, $1_{A_t}$ and $\ast_t$ respectively.
Suppose that $V \su A$ is a complex vector subspace such that
\begin{enumerate}
\item $\T{ev}_t \colon  V \to A_t$ is injective for all $t \in T$;
\item $1_A \in V$ and $V_t := \T{ev}_t(V)$ satisfies that $\xi^{*_t} \in V_t$ for all $\xi \in V_t$.
\end{enumerate}
For each $t \in T$, we then have the real part
\[
V_t^{\T{sa}} := \big\{ \xi \in V_t \mid \xi^{*_t} = \xi \big\} \su A_t^{\T{sa}},
\]
which becomes an order unit space with order unit $1_{A_t} \in V_t^{\T{sa}}$ and partial order coming from $A_t^{\T{sa}}$. For each $t \in T$, we fix a semi-norm $L_t \colon  V_t \to [0,\infty)$ satisfying that $L_t( 1_{A_t}) = 0$ and $L_t(\xi^{*_t}) = L_t(\xi)$ for all $\xi \in V_t$.
We emphasize the slight subtlety that our subspace $V \su A$ need not be invariant under the involution $*$ on $A$ even though the individual fibres $V_t \su A_t$ are invariant under the respective involutions $\ast_t$, so that we have an induced involution $*_t \colon  V \to V$ for every $t \in T$. The curious reader may cast a glance at the paragraph preceding Lemma \ref{lem:passage-to-band} to see how this situation arises naturally in the setting of crossed products.
For each $t \in T$, we denote the norm on $V_t$ coming from the $C^*$-norm on the fibre $A_t$ by $\| \cd \|_t \colon  V_t \to [0,\infty)$.
 Remark that the unital complex subspaces $V_t \su A_t$ need not be norm-dense. For an element $x \in A$, we apply the notation $x_t := \T{ev}_t(x)$ for all $t \in T$. 

\begin{dfn}\label{d:unicont}
We say that $\{(V_t^{\T{sa}},L_t^{\T{sa}})\}_{t \in T}$ is a \emph{uniform family of order unit compact quantum metric spaces} when 
\begin{enumerate}
\item $(V_t^{\T{sa}},L_t^{\T{sa}})$ is an order unit compact quantum metric space for all $t \in T$;
\item the map $t \mapsto L_t(y_t)$ is continuous on $T$ for all $y \in V$;
\item for each $t_0 \in T$ and $\ep > 0$ there exists an open neighbourhood $U \su T$ of $t_0$ such that
\[
(1 - \ep) L_{t_0}(y_{t_0} ) \leq L_t( y_t) \q \T{for all } y \in V \T{ and }t \in U;
\]
\item there exist a norm $\| \cd \|_\star \colon  V \to [0,\infty)$ and a constant $C > 0$ such that $\| \xi^{*_t} \|_\star = \| \xi \|_\star$ and 
\begin{equation}\label{eq:normfield}
\frac{1}{C} \cd \| \xi \|_\star \leq \| \xi_t \|_t \leq C \cd \| \xi \|_\star 
\end{equation}
for all $\xi \in V$ and all $t \in T$.
\end{enumerate}
\end{dfn}

Consider  a uniform family $\{ (V_t^{\T{sa}},L_t^{\T{sa}}) \}_{t \in T}$ of  order unit compact quantum metric spaces and fix a reference norm $\| \cd \|_\star \colon  V \to [0,\infty)$  and a constant $C > 0$ according to the definition. We may think of $\| \cd \|_\star$ as a norm on $V_t$ for each $t \in T$ via the $\cc$-linear isomorphism $\T{ev}_t \colon  V \to V_t$ and it then holds that 
\begin{equation}\label{eq:invpnt}
\| \xi_t \|_\star = \| \xi \|_\star = \| \xi_s \|_\star \q \T{for all }t, s \in T\T{ and }\xi \in V.
\end{equation}
For a fixed $r \geq 0$, we are interested in the two intersections of balls:
\[
\begin{split}
 \B B_1(L^{\T{sa}}_t) \cap \B B_r( \| \cd \|_\star) &:= \big\{ \xi \in V_t^{\T{sa}} \mid L_t^{\T{sa}}(\xi) \leq 1 , \| \xi \|_\star \leq r \big\} \\
 \B B_1(L_t) \cap \B B_r( \| \cd \|_\star) & := \big\{ \xi \in V_t \mid L_t(\xi) \leq 1 , \| \xi \|_\star \leq r \big\},
\end{split}
\]
as the parameter $t \in T$ varies. We are here explicitly writing out both these intersections because the difference between selfadjoint elements and general elements play an important role in the next lemma:
%

\begin{lem}\label{l:ballclos}
Suppose that $\{(V_t^{\T{sa}},L_t^{\T{sa}})\}_{t \in T}$ is a uniform family  of order unit compact quantum metric spaces and let $r \geq 0$ be given. For each $t_0 \in T$ and $\de > 0$ there exist finitely many elements $z^1,z^2,\ldots,z^N \in A^{\T{sa}}$ and an open neighbourhood $U \su T$ of $t_0$ such that
\begin{enumerate}
\item $z_t^n \in \B B_1(L_t^{\T{sa}}) \cap \B B_r( \| \cd \|_{\star})$ for all $t \in U$ and $n \in \{1,2,\ldots,N\}$;
\item $\B B_1(L_t^{\T{sa}}) \cap \B B_r( \| \cd \|_{\star}) \su \bigcup_{n = 1}^N \B B_\de( z_t^n, \| \cd \|_\star)$ for all $t \in U$.
\end{enumerate}
\end{lem}
\begin{proof}
Let $t_0 \in T$ and $\de > 0$ be given and assume without loss of generality that $\de < 2 r + 1$. Put $\ep := \frac{\de}{2 r + 1} \in (0,1)$. We notice that the norms $\| \cd \|_{t_0}$ and $\| \cd \|_\star \colon  V_{t_0} \to [0,\infty)$ are equivalent and that $*_{t_0} \colon  V_{t_0} \to V_{t_0}$ is an isometry for both $L_{t_0}$ and $\| \cd \|_\star$. Since $(V_{t_0}^{\T{sa}},L_{t_0}^{\T{sa}})$ is an order unit compact quantum metric space we may choose $f^1,\ldots,f^N \in V$ such that $f^1_{t_0},\ldots f^N_{t_0} \in \B B_1(L_{t_0}) \cap \B B_r( \| \cd \|_\star)$ and such that
\begin{equation}\label{eq:inc} 
\B B_1(L_{t_0}) \cap \B B_r(\| \cd \|_\star) \su \bigcup_{n = 1}^N \B B_\ep( f_{t_0}^n, \| \cd \|_\star) ,
\end{equation}
see \cite[Theorem 1.8]{Rie:MSA}. We stress that the intersection of balls above does not consist of selfadjoint elements. Using the conditions in Definition \ref{d:unicont}, we may choose an open neighbourhood $U \su T$ of $t_0$ such that
\begin{equation}\label{eq:lipcont}
L_t( f^n_t) \leq 1 + \ep \q \T{and} \q (1-\ep) L_{t_0}(y_{t_0}) \leq L_t(y_t)
\end{equation}
for all $t \in U$, $n \in \{1,2,\ldots,N\}$ and $y \in V$. For each $n \in \{1,2,\ldots,N\}$ and $t \in U$, define
\[
z_t^n := \frac{f^n_t + (f^n_t)^{*_t} }{2(1 + \ep)} \in V_t^{\T{sa}} .
\]
Using that $*_t \colon  V_t \to V_t$ is an isometry for both $L_t$ and $\| \cd \|_\star \colon  V_t \to [0,\infty)$ together with \eqref{eq:lipcont}, we may estimate as follows:
\[
\begin{split}
L_t(z_t^n) \leq \frac{L_t( f^n_t)}{1 + \ep} \leq 1 \q \T{and} \q
\| z_t^n \|_\star \leq \frac{\| f^n_t\|_\star}{1 + \ep} = \frac{\| f^n_{t_0} \|_\star }{1 + \ep} \leq r ,
\end{split}
\]
implying that $z_t^n \in \B B_1(L_t^{\T{sa}}) \cap \B B_r(\| \cd \|_\star)$ for all $t \in U$. Let now $t \in U$ and $y \in V$ with $y_t \in \B \B B_1(L_t^{\T{sa}}) \cap \B B_r(\| \cd \|_\star)$ be given. Using \eqref{eq:lipcont}, we see that
\[
L_{t_0}(y_{t_0}) \leq \frac{ L_t(y_t) }{1 - \ep} \leq \frac{1}{1 - \ep} \q \T{and} \q \| y_{t_0} \|_\star = \| y_t \|_\star \leq r ,
\]
implying that $(1 - \ep) \cd y_{t_0} \in \B B_1(L_{t_0}) \cap \B B_r(\| \cd \|_\star)$. We remark that $y_{t_0}$ need not be selfadjoint in the fibre $V_{t_0}$ (but that $y_t$ is selfadjoint in the fibre $V_t$). Using the inclusion in \eqref{eq:inc}, we may choose $j \in \{1,\ldots,N\}$ such that 
\[
\| (1 - \ep) y_{t_0} - f^j_{t_0} \|_\star \leq \ep .
\]
We then have the estimates
\[
\begin{split}
\| y_t - z^j_t \|_\star & = \big\| \frac{y_t + y_t^{*_t}}{2} - \frac{f^j_t + (f^j_t)^{*_t}}{2 (1 + \ep)} \big\|_\star
\leq \frac{\| y_t(1 + \ep) - f^j_t \|_\star}{2(1 + \ep)} + \frac{\| y_t^{*_t}(1 + \ep) - (f^j_t)^{*_t} \|_\star }{2(1 + \ep)} \\
& \leq 2 \ep \frac{\| y_t \|_\star}{1 + \ep} + \frac{\| y_t (1 - \ep) - f^j_t \|_\star}{1 + \ep}
\leq 2 \ep r + \frac{\| y_{t_0} (1 - \ep) - f^j_{t_0} \|_\star}{1 + \ep} \leq 2 \ep r + \ep = \de. 
\end{split}
\]
This proves the lemma.
\end{proof}

We may now state and prove the main theorem of this section:

\begin{thm}\label{t:GHconv}
Suppose that $\{(V_t^{\T{sa}},L_t^{\T{sa}})\}_{t \in T}$ is a uniform family of order unit compact quantum metric spaces and suppose that $\sup_{t \in T} r_t < \infty$, where $r_t \geq 0$ denotes the radius of the order unit compact quantum metric space $(V_t^\sa,L_t^{\T{sa}})$. Let $t_0 \in T$ be given and suppose that the unital $C^*$-algebra $A_{t_0}$ is separable. Then it holds that  
\[
\lim_{t \to t_0} \T{dist}_q( V_t^{\T{sa}}, V_{t_0}^{\T{sa}}) = 0 .
\]
\end{thm}
\begin{proof}
Since the injective $*$-homomorphism $j \colon  C(T) \to \C Z(A)$ is supposed to be unital it is also non-degenerate and we may thus identify $A$ with a $*$-subalgebra of $\prod_{t \in T} A_t$, see \cite[Proposition 2.8]{Bla:DCH}. We moreover see that $A$ is separating in the sense that if $t,s \in T$ with $t \neq s$ and $x \in A_t$, $y \in A_s$ are given, then there exists an $f \in A$ with $f_t = x$ and $f_s = y$. It therefore follows by \cite[Corollary to Theorem 1.4]{Fel:SAO} that $A$ is maximal. In other words, for each $\ga \in \prod_{t \in T} A_t$, we have the implication
\[
\Big( t \mapsto \| f_t - \ga_t \|_t \T{ is continuous on } T \T{ for all } f \in A \Big) \rar \big( \ga \in A \big).
\]
In particular, it holds that $A^{\T{sa}} \su \prod_{t \in T} A_t^{\T{sa}}$ defines a continuous field of real Banach spaces with fibres $\{ A_t^{\T{sa}} \}_{t \in T}$, see \cite[Section 10.1 and 10.2]{Dix:Calg} for details. Since the fibre $A^{\T{sa}}_{t_0}$ is separable we may apply \cite[Proposition 7.4]{Li:GH-dist} and choose a real Banach space $W$ together with isometric $\rr$-linear maps $\io_t\colon A_t^{\T{sa}} \to W$ such that it holds for every $f \in A^{\T{sa}}$ that the map $T \to W$ defined by $t \mapsto \io_t \circ f$ is continuous at $t_0 \in T$. Let us fix the reference norm $\| \cd \|_\star \colon  V \to [0,\infty)$ together with the norm-equivalence constant $C > 0$ as in Definition \ref{d:unicont}, and define $R := 1 + \sup_{t \in T} r_t$.  Let $\ep > 0$ be given and put $\de := \frac{\ep}{6 (C+1)} > 0$. By Lemma \ref{l:ballclos} we may choose $z^1,z^2,\ldots,z^N \in A^{\T{sa}}$ and an open neighbourhood $U \su T$ of $t_0 \in T$ such that
\begin{enumerate}
\item $z_t^n \in \B B_1(L_t^{\T{sa}}) \cap \B B_{C \cd R}( \| \cd \|_\star)$ for all $t \in U$ and $n \in \{1,2,\ldots,N\}$;
\item $\B B_1(L_t^{\T{sa}}) \cap \B B_{C \cd R}( \| \cd \|_\star) \su \bigcup_{n = 1}^N \B B_\de( z_t^n, \| \cd \|_\star)$ for all $t \in U$. 
\end{enumerate}
By shrinking $U$, if necessary, we may moreover assume that
\[
\| \io_t(1_{A_t})-  \io_{t_0}(1_{A_{t_0}})\| < \frac{\ep}{3 R} \q \T{and} \q \| \io_t(z_t^n)- \io_{t_0}(z_{t_0}^n) \|< \delta
\]
for all $t\in U$ and $n \in \{1,\dots, N\}$. For each $t \in U$ and $n \in \{1,2,\ldots,N\}$, we define
\[
\ze_t^n := \frac{R}{R + (C + 1) \de} \cd z_t^n \in V_t^\sa .
\]
Recall from \eqref{eq:ballinter} that $\C D_R(V_t^{\sa}) $ denotes the intersection $\B B_1(L_t^\sa) \cap \B B_R( \| \cd \|_t)$. We let \\ $\| \cd \| \colon  W \to [0,\infty)$ denote the norm on the real Banach space $W$ and claim that 
\[
\begin{split}
 \T{dist}_H^{\rho_{\| \cd \|}}\Big( \io_t\big( \C D_R(V_t^\sa)\big), \io_{t_0}\big( \C D_R(V_{t_0}^\sa) \big) \Big) < \frac{\ep}{3} \q \T{and} \q
 R \cd \| \io_t(1_{A_t}) - \io_{t_0}(1_{A_{t_0}}) \| < \frac{\ep}{3} 
\end{split}
\]
for all $t \in U$. Establishing this claim will prove the present theorem since Proposition \ref{p:lipghd} then entails that $\T{dist}_q(V_t^{\T{sa}}, V_{t_0}^{\T{sa}}) < \ep$ for all $t \in U$. 

Let thus $t \in U$ be given. By construction, we have that $R \cd \| \io_t(1_{A_t}) - \io_{t_0}(1_{A_{t_0}}) \| < \frac{\ep}{3}$ so we may focus on showing that
\[
\T{dist}_H^{\rho_{\| \cd \|}}\Big( \io_t\big( \C D_R(V_t^\sa) \big), \io_{t_0}\big( \C D_R(V_{t_0}^\sa) \big) \Big) < 2 (C + 1) \de  = \frac{\ep}{3}.
\]
Let $x_t \in \C D_R(V_t^\sa)$ be given. Since $\C D_R(V_t^\sa) \su \B B_1(L_t^{\T{sa}}) \cap \B B_{C \cd R}( \| \cd \|_\star)$ we may choose a $j \in \{1,2,\ldots,N\}$ such that $\| x_t - z_t^j \|_\star \leq \de$.

We notice that $\ze_{t_0}^j \in \C D_R(V_{t_0}^\sa)$. Indeed, clearly $\ze_{t_0}^j \in \B B_1(L_{t_0}^{\T{sa}})$ and the following estimate implies that $\| z_{t_0}^j\|_{t_0} \leq R + (C + 1) \de$ (recall that $\| x_t \|_t \leq R$):
\begin{equation}\label{eq:ballinc}
\begin{split}
\big| \| z_{t_0}^j \|_{t_0} - \| x_t \|_t \big| & \leq \big| \| z_{t_0}^j \|_{t_0} - \| z_t^j \|_t \big| + \| z_t^j - x_t \|_t \\
& \leq \| \io_{t_0}(z_{t_0}^j) - \io_t(z_t^j) \| + C \| z_t^j - x_t \|_\star
< \de + C \de .
\end{split}
\end{equation}
We moreover have that
\begin{equation}\label{eq:distance}
\begin{split}
\| \io_t(x_t) - \io_{t_0}(\ze_{t_0}^j) \| 
& \leq \| x_t - z_t^j \|_t + \| \io_t(z_t^j) - \io_{t_0}(z_{t_0}^j) \| + \|  z_{t_0}^j - \frac{R}{R + (C + 1)\de} z_{t_0}^j \|_{t_0} \\
& < C \de +\de  + \frac{(C + 1) \de}{ R + (C+1)\de} \| z_{t_0}^j \|_{t_0} \leq 2 (C + 1) \de .
\end{split}
\end{equation}
On the other hand, letting $y_{t_0} \in \C D_R(V_{t_0}^\sa)$ be given, we choose a $k \in \{1,2,\ldots,N\}$ such that $\| y_{t_0} - z_{t_0}^k \|_\star \leq \de$. Similar estimates to those given in \eqref{eq:ballinc} and \eqref{eq:distance} show that $\ze_t^k \in \C D_R(V_t^\sa)$ and that $\| \io_{t_0}(y_{t_0}) - \io_t(\ze_t^k) \| < 2 (C + 1)\de$. This proves our claim and hence the result of the theorem.
\end{proof}

\section{Quantum Gromov-Hausdorff convergence of crossed products}\label{sec:qgh-convergence-for-crossed-products}
In Section \ref{s:comcro} we gave criteria showing when an automorphism of a compact quantum metric space gives rise to a compact quantum metric structure on the associated crossed product algebra. Our next overall aim is to investigate what happens when we start varying the automorphism and what we can say about the corresponding variation of the corresponding crossed product compact quantum metric spaces with respect to the quantum Gromov-Hausdorff distance. To this end, we shall apply the general uniform criterion, which we developed in Section \ref{s:unicri}. \\ 

As in Section \ref{s:noniso} and Section \ref{s:comcro}, we consider a unital $C^*$-algebra $B$ together with a norm-dense $*$-invariant subspace $V_B \su B$ with $1 \in V_B$, and fix a semi-norm $L_B \colon  V_B \to [0,\infty)$ with $L_B(1) = 0$ and $L_B(\xi) = L_B(\xi^*)$ for all $\xi \in V_B$. We shall moreover fix an order preserving (see Definition \ref{d:vertiii}) norm $\vertiii{\cdot}\colon C_c(\ZZ) \to [0,\infty)$ which we assume normalized such that $\vertiii{e_0}=1$.

We are interested in the extra structure given by a compact Hausdorff space $T$ together with a $*$-automorphism $\be_t \colon  B \to B$ for each $t \in T$ subject to the condition that the map
\[
T \to B \q t \mapsto \be_t(b) 
\]
is continuous, with respect to the $C^*$-norm on $B$ for all $b \in B$. For each $t \in T$, we assume that $\be_t$ and $\be_t^{-1} \colon  B \to B$ preserve the subspace $V_B \su B$ so that
\[
\be_t(V_B) = V_B = \be_t^{-1}(V_B) .
\]
We may assemble the family $\{\beta_t\}_{t\in T}$ into a $*$-automorphism $\be \colon  C(T,B) \to C(T,B)$ defined by 
\[
\be(f)(t) := \be_t( f(t) )
\]
for all $f \in C(T,B)$ and $t \in T$, and this allows us to construct the reduced crossed product $\cro{C(T,B),\be}$. The commutative unital $C^*$-algebra $C(T)$ is naturally included into the center of this reduced crossed product 
\[
j \colon  C(T) \to \C Z\big( \cro{C(T,B),\be} \big)
\]
and this inclusion provides us with an upper semi-continuous field of $C^*$-algebras, see \cite{Rie:CFG}[Proposition 1.2]. We also refer to \cite{Bla:DCH} for more information on fields of $C^*$-algebras. In fact, since $\B Z$ is amenable we have an even stronger result, see \cite{Ped:CAA}[Theorem 7.7.7] and \cite{Rie:CFG}[Corollary 3.6]:

\begin{prop}\label{p:crossfield}
The inclusion $j \colon  C(T) \to \C Z( \cro{C(T,B),\be})$ gives $\cro{C(T,B),\be}$ the structure of a \emph{continuous} field of $C^*$-algebras and the fiber at each $t \in T$ identifies with the reduced crossed product $\cro{B,\be_t}$ via the quotient map
\[
\cro{C(T,B),\be} \longrightarrow \cro{B,\be_t}
\]
induced by the map $\T{ev}_t \colon  C(T,B) \to B$ evaluating at $t$.
\end{prop}

To ease the notation, we put
\[
A := \cro{C(T,B),\be} \q \T{and} \q A_t := \cro{B,\be_t}, 
\]
and denote the $C^*$-norm on $A_t$ by $\| \cd \|_t \colon  A_t \to [0,\infty)$ for all $t \in T$. For each $t \in T$, we have the semi-norms $L^1_t$ and $L^2_t \colon  C_c(\B Z,V_B) \to [0,\infty)$, defined in \eqref{eq:semi}, given by:
\[
\begin{split}
 L^1_t(\eta) &:=  \big\| \sum_{n=-\infty}^{\infty} n \cd \eta(n) U^n \big\|_t \q \T{and} \\
 L^2_t(\eta) &:= 
\max\big\{  \bvert{ L_B \ci \eta} , \bvert{ L_B \ci \eta^* } \big\} .
\end{split}
\]
Remark that, as a complex vector space, $V_t := V_{A_t}:= C_c(\B Z,V_B)$ is independent of $t \in T$. For each $t \in T$ we define $L_t := \max\{ L_t^1,L_t^2\}$ and recall from Theorem \ref{t:quacross} that $( \cro{B,\be_t},L_t)$ is a $C^*$-algebraic compact quantum metric space.

We are in this section establishing the following general convergence result:

\begin{thm}\label{t:crosconv}
With the setup as above, suppose that the following conditions hold:
\begin{enumerate}
\item the semi-norm $L_B \colon  V_B \to [0,\infty)$ is lower semi-continuous and $(B,L_B)$ is a compact quantum metric space;
\item for each $n \in \zz$ and each $\eta \in V_B$, the map $T \to [0,\infty)$ defined by $t \mapsto L_B( \be_t^n(\eta))$ is upper semi-continuous;
\item for each $n \in \zz$, $\ep > 0$ and $t_0 \in T$ there exists an open neighbourhood $U \su T$ of $t_0 \in U$ such that
\[
(1- \ep) L_B( \be_{t_0}^n(\eta))  \leq L_B(\be_t^n(\eta))
\]
for all $t \in U$ and $\eta \in V_B$. 
\end{enumerate}
Then it holds for every $t_0 \in T$ that
\[
\lim_{t \to t_0} \T{dist}_q\big( (\cro{B,\be_t}, L_t), (\cro{B,\be_{t_0}}, L_{t_0}) \big) = 0 .
\]
\end{thm}

\begin{remark}\label{rem:cont-rem}
Our standing assumption that $t\mapsto \beta_t(b)$ { is continuous} for all $b \in B$ implies that also $t\mapsto \beta_t^n(b)$ is continuous for all $n\in \zz$ and all $b\in B$. So if $L_B\colon V_B\to [0,\infty)$ is assumed lower semi-continuous (as in Theorem \ref{t:crosconv} $(1)$) then the map $t\mapsto L_B(\beta_t^n(\eta))$ is lower semi-continuous for each $\eta\in V_B$ and $n\in \zz$. Thus, assuming that condition $(1)$ in Theorem \ref{t:crosconv} holds, condition $(2)$ is equivalent with continuity of the map $t\mapsto L_B(\beta_t^n(\eta))$ for all $\eta \in V_B$ and all $n \in \zz$.
\end{remark}

The main part of the proof of Theorem \ref{t:crosconv} will be given in Subsection \ref{ss:procrocon} and relies heavily on the uniform criterion introduced in Section \ref{s:unicri} together with a technique allowing us to reduce the convergence problem to the case where there is an upper bound on the number of Fourier coefficients (in fact to the case where there is an upper bound on the distance to the diagonal). This latter technique applies to more general circle actions than the dual action on a crossed product by the integers and was found by Li, see \cite[Lemma 4.4]{Li:DCQ}, building on methods introduced by Rieffel, see \cite[Lemma 8.3 and 8.4]{Rie:GHD}.Ê\\

For the moment we illustrate the power of Theorem \ref{t:crosconv} by presenting two corollaries concerning quantum Gromov-Hausdorff convergence for isometric actions:

\begin{cor}\label{c:genisomet}
Let $(B,L_B)$ be a compact quantum metric space with $L_B \colon  V_B \to [0,\infty)$ lower semi-continuous and let $\{\beta_t\}_{t\in T}$ be a family of $*$-automorphisms of $B$ parametrized by a compact Hausdorff space $T$. Suppose that $\beta_t$ and $\beta_t^{-1}$ both preserve $V_B$ and that the map $t \mapsto \be_t(b)$ is continuous for every $b \in B$. Suppose moreover that each $*$-automorphism is $L_B$-isometric so that $L_B(\be_t(\xi)) = L_B(\xi)$ for all $\xi \in V_B$ and $t \in T$. Then it holds for every $t_0 \in T$ that
\[
\lim_{t \to t_0} \T{dist}_q\big( (\cro{B,\be_t}, L_t), (\cro{B,\be_{t_0}}, L_{t_0}) \big) = 0 .
\]
\end{cor}

To state the next corollary we need a few preliminaries. For a compact metric space $(X,d)$ we endow $C(X,X)$ with the metric
\[
d^\infty(\varphi,\psi):=\sup_{x\in X}d(\varphi(x),\psi(x)).
\]
We moreover define the dense complex subspace $V_{C(X)} \su C(X)$ by $V_{C(X)} :=\{f\in C(X)\mid f \text{ is Lipschitz continuous}\}$ 
together with the semi-norm 
\[
L_{C(X)}(f):=\sup\left \{ \frac{|f(x)-f(y)|}{d(x,y)} \mathrel{\Big|} x,y\in X, x\neq y\right\};
\]
i.e., $L_{C(X)}(f)$ is the minimal Lipschitz constant for $f$. We remark that $(C(X),L_{C(X)})$ is a compact quantum metric space and refer to \cite{Rie:MSS} for a discussion of this result. Being a supremum of continuous functions, the semi-norm $L_{C(X)} \colon  V_{C(X)} \to [0,\infty)$ is lower semi-continuous.


\begin{cor}\label{c:metisomet}
Let $(X,d)$ be a compact metric space and let $\{\psi_t\}_{t\in T}$ be a family of isometric homeomorphisms $\psi_t\colon X\to X$ parametrized by a compact Hausdorff space $T$ and satisfying that $T\ni t\mapsto \psi_t\in C(X,X)$ is continuous with respect to $d^\infty$. Then it holds for every $t_0 \in T$ that
\[
\lim_{t \to t_0} \T{dist}_q\big( (\cro{C(X),\be_t}, L_t), (\cro{C(X),\be_{t_0}}, L_{t_0}) \big) = 0 ,
\]
where $\beta_t\colon C(X)\to C(X)$ denotes the $*$-automorphism given by $\beta_t(f)=f\circ \psi_t$.
\end{cor}
\begin{proof}
We apply Corollary \ref{c:genisomet}. Since $\psi_t$ and $\psi_t^{-1}$ are isometric, both $\beta_t$ and $\beta_t^{-1}$ map $V_{C(X)}$ onto itself and it holds that $L_{C(X)}(\be_t(\xi)) = L_{C(X)}(\xi)$ for all $\xi\in V_B$ and $t\in T$. For a given $f \in C(X)$, it only remains to prove the continuity of the map $t\mapsto \beta_t(f)$. To see this, fix $t_0\in T$ and $\ep>0$. Since $X$ is compact, $f$ is uniformly continuous so there exists a $\delta > 0$ such that $|f(x)-f(y)|<\ep$ whenever $d(x,y)<\delta$. By assumption, there exists an open neighbourhood $U$ of $t_0$ such that ${d}^\infty(\psi_t, \psi_{t_0})<\delta$ for all $t\in U$. But then $\|\beta_t(f)-\beta_{t_0}(f)  \|_\infty =\sup_{x\in X} |f(\psi_t(x))-f(\psi_{t_0}(x))|<\ep$ for all $t\in U$. 
\end{proof}

We also present a different and perhaps more natural description of condition $(3)$ in Theorem \ref{t:crosconv}: One may ask whether a given linear map $\phi \colon  V_B \to V_B$ is bounded with respect to the semi-norm $L_B \colon  V_B \to [0,\infty)$ and in this case we have the operator semi-norm $\| \phi \|_{L_B} \in [0,\infty)$ defined by
\begin{equation}\label{eq:opesemi}
\| \phi \|_{L_B} := \sup\big\{  L_B( \phi(\eta)) \mid \eta \in V_B \, \, , \, \, \, L_B(\eta) \leq 1 \big\}.
\end{equation}
For each $n \in \zz$ and each $t,s \in T$, we define the linear isomorphism 
\[
\be_{t,s}^n := \be_t^n \ci \be_s^{-n} \colon  V_B \to V_B,
\]
and we will now show how condition (3) in Theorem \ref{t:crosconv} can be expressed  naturally in terms of the operator semi-norms of these maps.
\begin{lemma}\label{lem:operator-semi-norm-cont}
Let $n \in \zz$ and $t_0 \in T$ be given. With the assumptions as above, the following conditions are equivalent:
\begin{enumerate} 
\item there exists an open neighbourhood $U \su T$ of $t_0$ such that the linear isomorphism $\be_{t_0,t}^n \colon  V_B \to V_B$ is bounded for all $t \in U$ and the map $U \to [0,\infty)$ given by $t \mapsto \|\be_{t_0,t}^n\|_{L_B}$ is upper semi-continuous at $t_0 \in U$.
\item for each $\ep > 0$ there exists an open neighbourhood $U \su T$ of $t_0$ such that
\[
(1- \ep) L_B( \be_{t_0}^n(\eta))  \leq L_B(\be_t^n(\eta))
\]
for all $t \in U$ and $\eta \in V_B$.
\end{enumerate}
\end{lemma}

\begin{remark}
For fixed $n\in \zz$ and $\eta \in V_B$ the map $t\mapsto \beta_{t_0, t}^n(\eta)$ is continuous (compare e.g.~with Remark \ref{rem:cont-rem}) so if the semi-norm $L_B$ is  lower semi-continuous then so is the map $ t\mapsto L_B(\beta_{t_0, t}^n(\eta))$. Since a supremum of lower semi-continuous functions is again lower semi-continuous, it then follows that $t\mapsto \|\beta_{t_0,t}^n\|_{L_B}$ is lower semi-continuous. Thus, under the additional assumption that $L_B$ be lower semi-continuous, which is the case in many examples, upper semi-continuity of the map $t \mapsto \|\be_{t_0,t}^n\|_{L_B}$ actually coincides with continuity. 
\end{remark}


\begin{proof}
If $L_B$ is constantly zero the statement is trivial, so assume that this is not the case and $(1)$ holds. Let $1 > \ep > 0$ be given and put $\de := \frac{\ep}{1-\ep} > 0$. Using that $\be_{t_0,t_0}^n = \T{id}_{V_B} \colon  V_B \to V_B$ satisfies $\|\T{id}_{V_B}\|_{L_B}=1$ (since $L_B\neq 0)$ in combination with the upper semi-continuity of $t \mapsto \|\be_{t_0,t}^n\|_{L_B}$, we  find an open neighbourhood $U' \su T$ of $t_0$ such that $\| \be_{t_0,t}^n \|_{L_B} \leq 1 + \de$ for all $t \in U'$. For each $t \in U'$ and $\eta \in V_B$ we then estimate that
\[
(1 - \ep) L_B( \be_{t_0}^n(\eta)) = (1 - \ep) L_B\big( \be_{t_0,t}^n(\be_t^n(\eta)) \big) \leq (1 - \ep) \cd (1 + \de) \cd L_B(\be_t^n(\eta))
= L_B(\be_t^n(\eta)) .
\]
This shows that $(2)$ holds. Conversely, suppose that $(2)$ holds. 
Let $\ep > 0$ be given, put $\de := \frac{\ep}{1 + \ep} \in (0,1)$ and choose an open neighbourhood $U \su T$ of $t_0$ such that $(1 - \de) L_B(\be_{t_0}^n(\eta)) \leq L_B(\be_t^n(\eta))$ for all $\eta \in V_B$ and $t \in U$. For each $\ze \in V_B$ with $L_B(\ze) \leq 1$, we then estimate that
\[
L_B( \be_{t_0,t}^n(\ze)) = L_B\big( \be_{t_0}^n( \be_t^{-n}(\ze)) \big) \leq \frac{1}{1-\de} L_B(\ze) \leq \frac{1}{1-\de} = 1 + \ep . 
\]
This shows that $\| \be_{t_0,t}^n \|_{L_B} \leq 1 + \ep = \| \be_{t_0,t_0}^n \|_{L_B} + \ep$ for all $t \in U$ and hence that $(1)$ holds.
\end{proof}





\subsection{Continuous fields of compact quantum metric spaces}\label{ss:contfield}
In this subsection we show that the family of compact quantum metric spaces $\big\{ ( \cro{B,\be_t}, L_t)\big\}_{t \in T}$ fits with the already established notion of a continuous field of compact quantum metric spaces. Since this result is not needed anywhere else in the text, we shall not elaborate on the definition of a continuous field of (order unit) compact quantum spaces, but refer the reader to \cite[Definition 6.4]{Li:GH-dist}.

We start by investigating the continuity properties of the two semi-norms $L^1_t$ and $L^2_t$ as $t$ varies in the compact Hausdorff space $T$:


\begin{lem}\label{l:l1cont}
The map $t\mapsto L^1_t(\eta)$ is continuous for every $\eta \in C_c(\mathbb{Z}, V_B)$.
\end{lem}
\begin{proof}
Let $\eta= \sum_{n=-\infty}^\infty \xi_n U^n \in C_c(\B Z, V_B)$ and define $\pa(\eta) := i \cd \sum_{n=-\infty}^\infty n  \xi_n U^n \in C_c(\B Z,V_B)$. We have that
$L^1_t(\eta) = \| \pa(\eta) \|_t$ for all $t \in T$, and the desired continuity result therefore follows from Proposition \ref{p:crossfield}. 
\end{proof}

\begin{lem}\label{l:l2cont}
Suppose that the map $t \mapsto L_B(\be_t^n(\xi))$ is upper semi-continuous for every $\xi \in V_B$ and $n \in \zz$ and that the semi-norm $L_B \colon  V_B \to [0,\infty)$ is lower semi-continuous. Then the map $t \mapsto L^2_t(\eta)$ is continuous for every $\eta \in C_c(\B Z, V_B)$.
\end{lem}
\begin{proof}
For each $\xi \in V_B$ and $n \in \zz$ we saw in Remark \ref{rem:cont-rem} that the function $t \mapsto L_B(\be_t^n(\xi))$ is continuous on $T$. 

For fixed $\eta = \sum_{n = - \infty}^\infty \xi_n U^n \in C_c(\B Z,V_B)$ there exists an $N\in \nn$ such that $\eta = \sum_{n = - N}^N \xi_n U^n$,
and for $t \in T$ we may express the semi-norm $L^2_t(\eta) \in [0,\infty)$ as follows:
\[
 L^2_t(\eta) = \max \big\{ \bvert{ \sum_{n=-N}^N L_B(\xi_n) e_n}, \bvert{\sum_{n=-N}^N  L_B\big(\beta_t^n(\xi_{-n}^*)\big) e_n} \big\} . 
\]
We just saw that the coefficients $L_B\big(\beta_t^n(\xi_{-n}^*)\big)$ vary continuously in $t$, and since all norms on the finite dimensional space $\T{span}_{\cc}\{e_n \mid -N\leq n \leq N\}\su C_c(\ZZ)$ are equivalent the desired continuity of $t\mapsto L^2_t(\eta)$ follows.
\end{proof}

Recall that, as a vector space over $\cc$, we have that $V_t := C_c(\zz,V_B) \su A_t = \cro{B,\be_t}$ for all $t \in T$. 
\begin{prop}
Suppose that the map $t \mapsto L_B(\be_t^n(\xi))$ is upper semi-continuous for every $\xi \in V_B$ and $n \in \zz$ and that the semi-norm $L_B \colon  V_B \to [0,\infty)$ is lower semi-continuous.
Then the family $\{ (V_t^{\sa},L_t^{\sa})_{t\in T} \}$ together with the continuous field of real Banach spaces $\cro{C(T,B),\be}^{\sa} \su \prod_{t \in T} \cro{B,\be_t}^{\sa}$ form a continuous field of order unit compact quantum spaces in the sense of \cite[Definition 6.4]{Li:GH-dist}.
\end{prop}
\begin{proof}
Since the inclusion $j \colon C(T) \to \cro{C(T,B),\be}$ determines a continuous field of $C^*$-algebras by Proposition \ref{p:crossfield}, we already saw in the proof of Theorem \ref{t:GHconv} that the real part $\cro{ C(T,B),\be}^\sa$ identifies with the continuous sections of a continuous field of real Banach spaces with fibres $\{ \cro{B,\be_t}^{\sa} \}_{t \in T}$. Moreover, $\cro{ C(T,B),\be}^\sa$ clearly contains the unit section $1 = \{ 1_t \}_{t \in T}$. Thus, to prove the claim of the proposition we need, for given $t_0\in T$, $x\in V_{t_0}^{\sa}$ and $\ep>0$, to find a continuous section $\eta\in A^{\sa}$ such that $t\mapsto L_t^\sa(\eta_t)$ is upper semi-continuous at $t_0$ and such that $\|\eta_{t_0}-x\|_{t_0}<\ep$ and $L_{t_0}^\sa(\eta_{t_0})<L_{t_0}^\sa(x)+\ep$. 
However, for our given $x\in V_{t_0}^{\T{sa}}:=C_c(\zz, V_B )\cap  \cro{B,\be_{t_0}}^{\T{sa}}$ one may consider the element $\zeta\in C_c(\zz, C(T,B)) \su A$ given by $\zeta(n)(t)=x(n)$ and define $\eta := \tfrac{1}{2}(\zeta +\zeta^*) \in A^\sa$. The proof is therefore complete if we can show that $t\mapsto L_t(\eta_t)$ is upper semi-continuous at $t_0$. Writing $x=\sum_{n=-N}^N x(n) U^n$, we have that $\eta=\tfrac{1}{2}\sum_{n=-N}^N\big( x(n) +\beta_t^{n}(x(-n)^*) \big)U^n$, where the coefficients $x(n) +\beta_t^{n}(x(-n)^*)$ are now considered as elements in $C(T,B)$, and hence
\[
L_t^1(\eta_t)=\big\| \frac{1}{2}\sum_{n=-N}^N n \cd \big( x(n) +\beta_t^n(x(-n)^*) \big)U^n  \big\|_t .
\]
But, $\frac{1}{2}\sum_{n=-N}^N n \cd \big( x(n) +\beta_t^{n}(x(-n)^*) \big)U^n$ is still an element in $A = \cro{ C(T,B),\be}$ and it therefore follows from Proposition \ref{p:crossfield} that $t\mapsto L_t^1(\eta_t)$ is actually continuous (at every point in $T$).

So we may focus on showing that $t\mapsto L^2_t(\eta_t)$ is upper semi-continuous at $t_0$. Putting $C := \nvert{ \sum_{n = -N}^N e_n}$ we may use our assumption on upper semi-continuity to arrange that $L_B\big(\beta_t^n(x(-n)^*))\big) \leq L_B\big(\beta_{t_0}^n(x(-n)^*))\big) + \frac{2 \ep}{C}$ for all $n \in \{-N,\ldots,N\}$ and all $t$ in a suitable open neighbourhood $U$ of $t_0$. For each $t \in U$, we then estimate that 
\begin{align*}
L^2_t(\eta_t) &= \frac12 \bvert{\sum_{n=-N}^N L_B\big(x(n) +\beta_t^n(x(-n)^*)\big)e_n  }\\
&\leq \frac12 \bvert{  \sum_{n=-N}^N\big(L_B(x(n)) + L_B(\beta_t^n(x(-n)^*))\big)e_n  } \\
& \leq \frac12\bvert{  \sum_{n=-N}^N\big(L_B(x(n)) + L_B(\beta_{t_0}^n(x(-n)^*)) + \frac{2 \ep}{C} \big)e_n  } \\
& \leq \bvert{  \sum_{n=-N}^N L_B(x(n)) e_n  } +\ep
= L^2_{t_0}(\eta_{t_0}) +\ep,
\end{align*}
where we used that $\nvert{\cd}$ is order preserving and that $\eta_{t_0}=x$ is selfadjoint so that $x(n)=\beta_{t_0}^{n}(x(-n)^*)$ for all $n \in \zz$. This establishes the claimed upper semi-continuity result.
\end{proof}





\subsection{The convergence result for crossed products}\label{ss:procrocon}
We are now ready to embark for real on the proof of Theorem \ref{t:crosconv}. Our first step is to investigate the convergence problem when we have an upper bound on the distance to the diagonal. More precisely, we fix an $N \in \nn_0 $ and define the complex subspace $V^N \su \cro{C(T,B),\be}$ by
\[
V^N := \big\{ \xi \in C_c(\zz,C(T,B)) \mid \xi(n) = 0 \T{ for } |n| > N \T{ and } \xi(n) \in i(V_B) \T{ for } |n| \leq N \big\},
\]
where the $*$-homomorphism $i \colon  B \to C(T,B)$ is given by $i(b)(t) := b$. We define the norm
\[
\| \cd \|_{\star} \colon  V^N \to [0,\infty) \q \| \xi \|_{\star} := \max_{|n| \leq N} \| \xi(n) \|_{C(T,B)} .
\]

In the next two lemmas we verify the preliminary conditions for having a uniform family of order unit compact quantum metric spaces as in Definition \ref{d:unicont}.

\begin{lemma}\label{l:starstar}
The evaluation map $\T{ev}_t \colon  V^N \to C_c(\zz,V_B)$ is injective and for each $t \in T$, $\T{ev}_t(V^N) =: V^N_t$ is invariant under $*_t \colon  \cro{B,\be_t} \to \cro{B,\be_t}$. Moreover, we have that $\| \xi^{*_t} \|_\star = \| \xi \|_\star$, where the involution $*_t \colon  V^N \to V^N$ is inherited from $*_t \colon  V^N_t \to V^N_t$ for all $t \in T$.
\end{lemma}
\begin{proof}
{Since the Fourier coefficients of elements in $V^N$ are constant functions with values in $B$ it is clear that $\T{ev}_t \colon  V^N \to V^N_t$ is injective. Moreover, since $V_B$ is stable under the action of $\be_t$ and $\be_t^{-1}$ and under the adjoint operation it follows from the structure of $V^N$ that the involution $*_t$ preserves the subspace $V^N_t$.} Let now $\xi \in V^N$ be given. For each $n \in \{-N,\ldots,N\}$ choose $\eta_n \in V_B$ with $\xi(n) = i(\eta_n)$. We compute that
\[
\begin{split}
\| \xi^{*_t} \|_\star & = \big\| \sum_{n = -N}^N i( \be_t^{-n}(\eta_n^*)) U^{-n} \big\|_\star
= \max_{|n| \leq N} \| \be_t^{-n}(\eta_n^*) \|_B
 = \max_{|n| \leq N} \| \eta_n \|_B = \| \xi \|_\star . \qedhere
\end{split}
\]
\end{proof}

\begin{lemma}\label{l:starineq}
Let $N \in \nn_0 $. We have the inequalities
\[
\begin{split}
 \| \xi \|_\star \leq \| \xi_t \|_t \leq (2N + 1) \| \xi \|_\star \q \mbox{and}  \q L^1_t(\xi) \leq 2N^2 \| \xi \|_\star 
\end{split}
\]
for all $\xi \in V^N$ and all $t \in T$.
\end{lemma}
\begin{proof}
Write $\xi = \sum_{n = -N}^N i(\eta_n) U^n$ for $\eta_n \in V_B$ and let $t \in T$ be given. For each $n \in \{-N,\ldots,N\}$, we have that
\[
\| i(\eta_n) \|_{C(T,B)} = \| \eta_n \|_B = \| \tau_0(\xi_t \cd (U^*)^n) \|_B \leq \| \xi_t \cd (U^*)^n \|_t = \| \xi_t \|_t,
\]
where $\tau_0 \colon  \cro{B,\be_t} \to B$ is the conditional expectation introduced right before Lemma \ref{l:lowsemi}. This shows that $\| \xi \|_\star \leq \| \xi_t \|_t$. The remaining inequalities follow by an application of the triangle inequality.
\end{proof}

The next proposition establishes our convergence result in the setting where there is an upper bound on the distance to the diagonal:

\begin{prop}\label{p:GHband}
Suppose that the following conditions hold:
\begin{enumerate}
\item The semi-norm $L_B \colon  V_B \to [0,\infty)$ is lower semi-continuous and $(B,L_B)$ is a compact quantum metric space;
\item For each $n \in \zz$ and each $\eta \in V_B$, the map $T \to [0,\infty)$ defined by $t \mapsto L_B( \be_t^n(\eta))$ is upper semi-continuous;
\item For each $n \in \zz$, $\ep > 0$ and $t_0 \in T$ there exists an open neighbourhood $U \su T$ of $t_0$ such that
\[
(1- \ep) L_B( \be_{t_0}^n(\eta))  \leq L_B(\be_t^n(\eta))
\]
for all $t \in U$ and $\eta \in V_B$.

\end{enumerate}
Then $\big\{ \big(\left(V_t^N\right)^{\T{sa}},L_t|_{\left(V_t^N\right)^{\T{sa}}}\big) \big\}_{t \in T}$ is a uniform family  of order unit compact quantum metric spaces. In particular, we have that
\[
\lim_{t \to t_0} \T{dist}_q\big( \left(V_t^N\right)^\sa, \left(V_{t_0}^N\right)^\sa \big) = 0  \q \mbox{for all } t_0 \in T .
\]
\end{prop}
\begin{proof}
We start by verifying the four conditions in Definition \ref{d:unicont}. To verify condition $(1)$ we apply Theorem \ref{t:quacross} which shows that each fibre $(\cro{B,\be_t},L_t)$ is a compact quantum metric space and this implies that the real part is an order unit compact quantum metric space. Condition $(1)$ then follows since $(V^N_t)^{\T{sa}} \su V_t^{\T{sa}}$, see \cite[Theorem 1.8]{Rie:MSA}. We already verified condition $(4)$ in Lemma \ref{l:starstar} and Lemma \ref{l:starineq} and condition $(2)$ follows from Lemma \ref{l:l1cont} and Lemma \ref{l:l2cont}, so we focus on condition $(3)$. Let $1 > \ep > 0$ and $t_0 \in T$ be given. We choose an open neighbourhood $U \su T$ of $t_0$ such that
\[
(1 - \ep) L_B( \be_{t_0}^n(\eta))  \leq L_B(\be_t^n(\eta))
\]
for all $n \in \{-N,\ldots,N\}$ and all $\eta \in V_B$. We thus have that
\[
\begin{split}
 (1 - \ep) L^2_{t_0}\big( \sum_{n = -N}^N \eta_n U^n\big) & = (1 - \ep) \cd \max\big\{ \bvert{ \sum_{n = -N}^N L_B(\eta_n) e_n} 
, \bvert{ \sum_{n = -N}^N L_B(\be^n_{t_0}(\eta_{-n}^*)) e_n} \big\} \\
& \leq \max\big\{ (1 - \ep) \bvert{ \sum_{n = -N}^N L_B(\eta_n) e_n} , \bvert{ \sum_{n = -N}^N L_B(\be^n_t(\eta_{-n}^*)) e_n} \big\} \\
&  \leq L^2_t\big( \sum_{n = -N}^N \eta_n U^n\big) 
\end{split}
\]
for all $\sum_{n = -N}^N i(\eta_n) U^n \in V^N$. \\ 
Next, put $\de := \frac{\ep}{4N^2 + 1} > 0$. Since the fibre $( \cro{B,\be_{t_0}},L_{t_0})$ is a compact quantum metric space we may find finitely many $\xi^1,\ldots, \xi^M \in V^N$ with $L_{t_0}(\xi^j_{t_0}) \leq 1$ for all $j \in \{1,2,\ldots,M\}$ such that
\[
q_{t_0}( \B B_1(L_{t_0}) \cap V_{t_0}^N) \su \bigcup_{j = 1}^M \B B_\de\left( q_{t_0}(\xi^j_{t_0}), \| \cd \|_\star^{\sim} \right),
\]
where $q_{t_0} \colon  V^N_{t_0} \to V^N_{t_0} / \cc 1_{A_{t_0}}$ is the quotient map and $\| \cd \|_\star^{\sim}$ is the quotient norm. By shrinking $U$ if necessary, we may assume that
\begin{equation}\label{eq:lipcont2}
L_t^1(\xi_t^j) \in [ L_{t_0}^1(\xi_{t_0}^j) - \de, L_{t_0}^1(\xi_{t_0}^j) + \de] 
\end{equation}
for all $t \in U$ and all $j \in \{1,2,\ldots,M\}$, see Lemma \ref{l:l1cont}.
Next, let $\xi \in V^N$ with $L_{t_0}(\xi_{t_0}) = 1$ be given. Choose $j \in \{1,2,\ldots,M\}$ such that $\| q_{t_0}(\xi_{t_0} - \xi^j_{t_0}) \|_\star^{\sim} \leq \de$ and remark that this implies that $\| q_t(\xi_t - \xi_t^j) \|_\star^{\sim} \leq \de$ for all $t \in T$, see \eqref{eq:invpnt}. 
Using Lemma \ref{l:starineq} and \eqref{eq:lipcont2} we now estimate as follows:
\[
\begin{split}
L_{t_0}^1(\xi_{t_0}) 
& \leq L_{t_0}^1(\xi_{t_0} - \xi_{t_0}^j) + L_{t_0}^1(\xi_{t_0}^j) 
\leq 2N^2 \cd \| q_{t_0}(\xi_{t_0} - \xi^j_{t_0}) \|_\star^{\sim} 
+ L_{t_0}^1(\xi_{t_0}^j) \\
& \leq 2N^2 \de + L_t^1(\xi_t^j) + \de
\leq 2N^2 \de + L_t^1(\xi_t^j - \xi_t) + L_t^1(\xi_t) + \de \\
& \leq 2N^2 \de + 2N^2 \cd \| q(\xi_t - \xi^j_t) \|_\star^{\sim} + L_t^1(\xi_t) + \de
\leq 4N^2 \de + L_t^1(\xi_t) + \de
= L_t^1(\xi_t) + \ep 
\end{split}
\]
for all $t \in U$. We thus have that
\[
1 = L_{t_0}(\xi_{t_0}) = \max\{ L_{t_0}^1(\xi_{t_0}), L_{t_0}^2(\xi_{t_0}) \}
\leq \max\big\{ \ep + L_t^1(\xi_t), \frac{L_t^2(\xi_t)}{1 - \ep} \big\}
\]
for all $t \in U$. But this implies that
\[
\begin{split}
(1 - \ep) L_{t_0}(\xi_{t_0}) = 1 - \ep \leq  
\max\{ L_t^1(\xi_t), L_t^2(\xi_t)\} = L_t(\xi_t)
\end{split}
\]
for all $t \in U$. This proves that $\{ ((V_t^N)^{\T{sa}},L_t|_{V^N_t}^{\T{sa}}) \}_{t \in T}$ is a uniform family of order unit compact quantum metric spaces. \\
To end the proof of the proposition, we remark that by \cite[Theorem 4.3.4]{KaRi:FTO} we have an isometric embedding $B^{\T{sa}} \to C(S(B), \rr)$ and since $B$ is assumed to be a compact quantum metric space its state space is metrizable, and thus separable by compactness, and hence also $C(S(B), \rr)$ is separable. From this it follows that the $C^*$-algebra $\cro{B,\be_t}$ is separable for every $t \in T$. The remaining statement about convergence in the quantum Gromov-Hausdorff distance now follows by Theorem \ref{t:GHconv} provided that we can find an upper bound on the radii $\{ r_t^N \}_{t \in T}$ of the involved order unit compact quantum metric spaces, $\big\{ ( (V_t^N)^{\T{sa}}, L_t|_{(V_t^N)^{\T{sa}}} ) \big\}_{t \in T}$. 
However, letting $r_t \geq 0$ denote the radius of $( \cro{B,\be_t},L_t)$ we know from Theorem \ref{t:quacross} that $r_t \leq r_B + \frac{\pi}{2}$ and hence by Proposition \ref{p:radbou} it also holds that $r_t^N \leq r_B + \frac{\pi}{2}$ for all $t \in T$. This provides an upper bound on the involved radii.
\end{proof}

We now explain how to pass from the general convergence problem of Theorem \ref{t:crosconv} to the convergence problem, where we have an upper bound on the distance to the diagonal. This step in combination with Proposition \ref{p:GHband} will then allow us to provide a proof of Theorem \ref{t:crosconv}.
Define the unital subspace $V \su \cro{C(T,B),\be}$ by
\[
V := \{ \xi \in C_c(\zz,C(T,B)) \mid \xi(n) \in i(V_B) \T{ for all } n \in \zz \},
\]
where we recall that $i \colon  B \to C(T,B)$ embeds $B$ as constant $B$-valued maps on $T$. We let $V_t \su \cro{B,\be_t}$ denote the image of $V$ under the evaluation map $\T{ev}_t \colon  \cro{C(T,B,\be)} \to \cro{B,\be_t}$ and remark that $V_t = C_c(\zz,V_B)$ for all $t \in T$, as a vector space over $\cc$.\\
For $t \in T$ and $R \geq 0$, we recall the notation
\[
\C D_R(V_t^\sa) := \{ \xi \in V_t^\sa \mid L_t(\xi)\leq 1, \|\xi\|_t\leq R \} .
\]
The following lemma shows how quantum Gromov-Hausdorff continuity at the level of the bands $\left\{\left(V_t^N\right)^{\T{sa}}\right\}_{t\in T}$ can be lifted to  continuity of 
$\left\{V_t^{\T{sa}}\right\}_{t\in T}$.
\begin{lem}\label{lem:passage-to-band}
For every $\ep>0$ there exists an $N\in \nn_0 $ such that $\T{dist}_q\big(V_t^{\sa}, (V_t^N)^\sa\big)<\ep$ for all $t\in T$.
\end{lem}
\begin{proof}
In Theorem \ref{t:quacross} we showed that $( \cro{B,\be_t}, L_t)$ is a compact quantum metric space for all $t \in T$, by showing that the conditions in \cite[Theorem 4.1]{Li:DCQ} are satisfied with constant $C=1$. Notice also that the length function $l \colon  \B T \to [0,\infty)$ (given by arc length) is independent of $t \in T$. In view of the upper bound on the radii, $r_t \leq r_B + \frac{\pi}{2}$ for all $t \in T$, we put $R := r_B + \frac{\pi}{2}$.
For a given $\ep>0$, \cite[Lemma 4.4]{Li:DCQ}, now provides us with an $N \in \nn_0 $ such that for all $t\in T$ and $\xi \in \C D_R(V_t^\sa)$ there exists $\xi'\in \C D_R((V_t^N)^\sa)$ with $\|\xi-\xi'\|_t < \frac{\ep}{3}$. Since $\C D_R((V_t^N)^\sa) \su \C D_R(V_t^\sa)$ this shows that $\T{dist}_H^{\rho_{\| \cd \|_t}}\big( \C D_R((V_t^N)^\sa), \C D_R(V_t^\sa) \big)< \frac{\ep}{3}$. Since the embeddings $(V_t^N)^\sa \su V_t^\sa \su \cro{B,\be_t}$ are unital (and isometric) we conclude that $\T{dist}_{\T{oq}}^R( (V_t^N)^\sa, V_t^{\sa}) < \frac{\ep}{3}$ for all $t \in T$. The result of the lemma now follows from Proposition \ref{p:lipghd}.
\end{proof}

\begin{proof}[Proof of Theorem \ref{t:crosconv}]
Fix $t_0\in T$ and $\ep>0$. We need to find an open neighbourhood $U\su T$ of $t_0 \in T$ such that for all $t\in U$ we have $\T{dist}_q(V_t^{\sa}, V_{t_0}^{\sa})<\ep$. By Lemma \ref{lem:passage-to-band} there exists an $N \in \nn_0 $ such that 
$\T{dist}_q(V_t^{\sa}, (V_t^N)^{\sa})<\frac{\ep}{3}$ for all $t\in T$. Moreover, since the assumptions in Theorem \ref{t:crosconv} are identical with those of Proposition \ref{p:GHband}, we obtain an open neighbourhood $U$ of $t_0$ such that $\T{dist}_q((V_t^N)^{\sa}, (V_{t_0}^N)^{\sa})< \frac{\ep}{3}$ for all $t \in U$. The estimate $\T{dist}_q(V_t^{\sa}, V_{t_0}^{\sa})<\ep$ now follows from the triangle inequality for the quantum Gromov-Hausdorff distance, see \cite[Theorem 4.3]{Rie:GHD}.
\end{proof}

\section{Diffeomorphisms of Riemannian manifolds}\label{sec:diff-of-riemannian}

Throughout this section $M$ denotes a connected, compact Riemannian manifold. We are going to investigate the quantum Gromov-Hausdorff convergence problem for crossed products by the integers related to diffeomorphisms of the manifold $M$. The setting will be a very general one where we do not put any restrictions on the diffeomorphisms we consider, in particular they are not assumed to be related to the Riemannian metric in any way. Relying on the results of the previous sections, for any such diffeomorphism we show how the corresponding crossed product becomes a compact quantum metric space and moreover we show that these crossed products vary continuously in the quantum Gromov-Hausdorff distance, whenever we have a family of diffeomorphisms $T \to \T{Diff}(M)$ which is continuous in the Whitney $C^1$-topology. We emphasize that, contrary to earlier approaches to crossed products by non-isometric actions, see for example \cite[Part I]{CoMo:LIF} and \cite[Section 4]{BMR:DSS}, we do not need to change the algebra of coordinates. This is particularly important when dealing with compact quantum metric spaces, since the substitutes in \cite{CoMo:LIF,BMR:DSS} are always non-compact and hence display a quite different behaviour from a quantum metric perspective, \cite{Lat:BLD,Lat:QLC,MeRe:NMU}.\\

We let $T_\rr M \to M$ denote the tangent bundle and $T M \to M$ denote the complexified tangent bundle. For each $p \in M$, we apply the notation $\inn{\cd,\cd}_p \colon  T_p M \ti T_p M \to \cc$ for the inner product on the fibre $T_p M$ coming from the Riemannian metric and in this way $T_p M$ becomes a finite dimensional complex Hilbert space. We apply the notation $\| \cd \|_p \colon  T_p M \to [0,\infty)$ for the corresponding Hilbert space norm. For any smooth function $f \in C^\infty(M)$ and each $p \in M$, we let $df(p) \colon  T_p M \to \cc$ denote the exterior derivative of $f$ at the point $p \in M$. 
Comparing with Section \ref{s:noniso} we now put $B = C(M)$, $V_B = C^\infty(M)$ and define the semi-norm $L_{C(M)} \colon  C^\infty(M) \to [0,\infty)$ by $L_{C(M)}(f) := \sup_{p \in M}\| df(p) \|_\infty$ for all $f \in C^\infty(M)$, where $\| \cd \|_\infty \colon  \B L( T_p M, \cc ) \to [0,\infty)$ denotes the operator norm. Let $d \colon  M \ti M \to [0,\infty)$ denote the metric on $M$ coming from the Riemannian metric, meaning that
\begin{equation}\label{eq:metrie}
d(p,q) = \inf\left\{ \int_a^b \| d \ga(t) \|_{\ga(t)}  dt \mathrel{\Big|} \ga \colon  [a,b] \to M \T{ smooth, } \ga(a) = p, \ga(b) = q \right\}. 
\end{equation}
In particular, for any smooth map $f \colon  M \to \cc$ we have the Lipschitz constant $C_f$ defined by
\[
C_f := \sup \left\{ \frac{|f(p) - f(q)| }{ d(p,q)} \mathrel{\Big|} p \neq q \right\}.
\]
Remark that $0 \leq C_f \leq L_{C(M)}(f) < \infty$ since
\begin{equation}\label{eq:lipuppmet}
\begin{split}
|f(p) - f(q)| & = | \int_a^b (f \ci \ga)'(t) dt | \leq \int_a^b \| df(\ga(t)) \|_\infty \| d\ga(t) \|_{\ga(t)} dt \\
& \leq L_{C(M)}(f) \int_a^b \| d\ga(t) \|_{\ga(t)} dt
\end{split}
\end{equation}
for every $p,q \in M$ and every smooth curve $\ga \colon  [a,b] \to M$ with $\ga(a) = p$ and $\ga(b) = q$.\\

The following result can be found in \cite[Proposition 1]{Con:CFH} in the context of spin manifolds, but is known to be valid in the much more general context of Lipschitz manifolds, see \cite[Remark 2]{Con:CFH}. For the sake of completeness, we present a few extra details on the proof:

\begin{prop}\label{p:quametrie}
Suppose that $M$ is a connected, compact Riemannian manifold. For each $f \in C^\infty(M)$, we have the identity $L_{C(M)}(f) = C_f$. In particular, it holds that $(C(M),L_{C(M)})$ is a compact quantum metric space and the corresponding metric $d_{L_{C(M)}}$ on $M \su S(C(M))$ agrees with the metric $d \colon  M \ti M \to [0,\infty)$ defined in \eqref{eq:metrie}.
\end{prop}
\begin{proof}
Once we have established that $L_{C(M)}(f) = C_f$ for all $f \in C^\infty(M)$, the remaining claims of the proposition follow from the work of 
Kantorovi\v{c} and Rubin\v{s}te\u{\i}n, \cite{KaRu:FSE,KaRu:OSC}, so we focus on computing the semi-norm $L_{C(M)} \colon  C^\infty(M) \to [0,\infty)$. Let $f \in C^\infty(M)$ be given. From the computation in \eqref{eq:lipuppmet} it follows that $C_f \leq L_{C(M)}(f)$. On the other hand, for each $p \in M$ and each unit vector $\xi \in T_p M$, we may choose a geodesic $\ga \colon  (-\ep,\ep) \to M$ with $\ga(0) = p$ and $d\ga(0) = \xi$, see \cite[Theorem 4.27]{Lee:IRM}. Since geodesics are locally length minimizing, see \cite[Theorem 6.15]{Lee:IRM}, we may assume that 
\[
l(s) := d( \ga(s),p) = \int_0^s \| d\ga(t) \|_{\ga(t)} dt
\]
for all $s \in [0,\ep)$. Note that the function $l(s) := d(\ga(s),p)$ is differentiable at $s = 0$ with derivative $l'(0) = \| d\ga(0) \|_p = \| \xi \|_p = 1$. We then estimate that
\[
\begin{split}
| df(\xi) | & = | (f \ci \ga)'(0)| = \lim_{h \to 0^+} \left| \frac{f(\ga(h)) - f(p)}{h} \right| \\
& = \lim_{h \to 0^+} \left| \frac{f(\ga(h)) - f(p)}{l(h)} \cd \frac{l(h)}{h}\right| \leq C_f \cd \lim_{h \to 0^+} \frac{l(h)}{h} = C_f .
\end{split}
\]
This shows that $L_{C(M)}(f) \leq C_f$, and the proof is complete.
%
\end{proof}

For any diffeomorphism $\psi \colon  M \to M$ we obtain a $*$-automorphism $\be_\psi \colon  C(M) \to C(M)$, $f \mapsto f \ci \psi$ which restricts to an isomorphism $\be_\psi \colon  C^\infty(M) \to C^\infty(M)$.  Fixing an order preserving, normalized norm $\vertiii{\cdot}$ on $C_c(\zz)$ we may, as in Section \ref{s:noniso},  consider the reduced crossed product $A := \cro{C(M), \be_\psi}$ together with the semi-norm $L_\psi := \max\{L_1,L_2\} \colon  C_c(\zz,C^\infty(M)) \to [0,\infty)$ defined by the components
\[
\begin{split}
& L_1\big( \sum_{n = -\infty}^\infty f_n U^n\big) := \big\| \sum_{n = -\infty}^\infty n \cd f_n U^n \big\|_A \q \T{and} \\
& L_2\big( \sum_{n = -\infty}^\infty f_n U^n \big) := \max\big\{ \bvert{\sum_{n=-\infty}^\infty L_{C(M)}(f_n)e_n } , \bvert{\sum_{n=-\infty}^\infty L_{C(M)}\left( \ov{f_{-n}} \ci \psi^{\ci n}\right)e_n } \big\}
\end{split}
\]
for all $\sum_{n = -\infty}^\infty \xi_n U^n \in V_A := C_c(\zz,C^\infty(M))$. \\

The next result is now an immediate consequence of Theorem \ref{t:quacross}, once we note that the semi-norm $L_{C(M)} \colon  C^\infty(M) \to [0,\infty)$ is indeed lower semi-continuous, which follows from Proposition \ref{p:quametrie},  since 
\[
L_{C(M)}(f) = C_f = \sup\left\{ \frac{|f(p) - f(q)|}{d(p,q)} \mathrel{\Big|} p \neq q \right\}\q \T{for all }f \in C^\infty(M) .
\]
\begin{thm}
Suppose that $\psi \colon  M \to M$ is a diffeomorphism of a connected, compact Riemannian manifold. Then the semi-norm $L_\psi \colon  C_c(\zz,C^\infty(M)) \to [0,\infty)$ is lower semi-continuous and $\big( \cro{C(M),\be_\psi},L_\psi\big)$ is a compact quantum metric space.
\end{thm}

Our next aim is to investigate the quantum Gromov-Hausdorff continuity of the crossed product as the diffeomorphism varies. To this end,
 we fix a compact Hausdorff space $T$ together with a map $\psi \colon  T \to \T{Diff}(M)$ associating a diffeomorphism $\psi_t \colon  M \to M$ to each $t \in T$. In particular, putting $\be_t := \be_{\psi_t}$ and $L_t := L_{\psi_t}$ we obtain a compact quantum metric space $\big( \cro{C(M),\be_t},L_t\big)$ for all $t \in T$. We are interested in the continuity properties of the map
\[
t \mapsto \big( \cro{C(M),\be_t},L_t\big)
\]
with respect to the quantum Gromov-Hausdorff distance. In order to investigate this problem we recall the definition of the Whitney $C^1$-topology on $C^\infty(M,M)$, see for example \cite[Chapter 2.1]{Hir:DT}. For a smooth map $f \colon  M \to M$, coordinate charts $(U,x)$ and $(V,y)$ on $M$, and any point $p \in U$ with $f(p) \in V$, we consider the Jacobian $J_{y,x}(f)(p) \in \B M_{\dim(M)}(\cc)$ defined by
\[
\big( J_{y,x}(f)(p) \big)_{ij} := \frac{\pa (y_i \ci f)}{\pa x_j}\Big|_p  \colon  U \to \cc  \q \T{for all } i,j \in \{1,\ldots,\dim(M)\}.
\]
Let now $f\in C^\infty(M,M)$ be fixed. Whenever $(U,x)$ and $(V,y)$ are coordinate charts on $M$, $K \su U$ is a compact set with $f(K) \su V$, and $\ep > 0$, we define the subset $W_f(x,y,K,\ep) \su C^\infty(M,M)$ consisting of those $g \in C^\infty(M,M)$ satisfying that
\begin{enumerate}
\item $\sup_{p \in K} d(f(p),g(p)) < \ep$;

\item $g(K) \su V$ and $\sup_{p \in K} \| J_{y,x}(f)(p) - J_{y,x}(g)(p)\|_\infty < \ep$ ,
\end{enumerate}
where $\| \cd \|_\infty \colon  \B M_{\dim(M)}(\cc) \to [0,\infty)$ denotes the operator norm, and $d$ denotes the Riemannian metric defined in \eqref{eq:metrie}. A neighbourhood basis around $f$ for the Whitney $C^1$-topology is then provided by taking finite intersections of the subsets $W_f(x,y,K,\ep) \su C^\infty(M,M)$ while varying the coordinate charts, the compact subset, and the parameter $\ep > 0$. We consider the diffeomorphism group of $M$, $\T{Diff}(M) \su C^\infty(M,M)$, endowed with the subspace topology coming from the Whitney $C^1$-topology, and remark that $\T{Diff}(M)$ in this way becomes a topological group; see for example \cite[Corollary 3.1.12]{Kup:DGM}. We can now announce the following main theorem,  the proof of which will occupy the remainder of this { section.}

\begin{thm}\label{t:convdiff}
Suppose that $M$ is a connected, compact Riemannian manifold, that $T$ is a compact, Hausdorff space and that $\psi \colon  T \to \T{Diff}(M)$ is continuous with respect to the Whitney $C^1$-topology. Then
\[
\lim_{t \to t_0} \T{dist}_q\big( (\cro{C(M),\be_t}, L_t), (\cro{C(M),\be_{t_0}}, L_{t_0})\big) = 0 
\]
for all $t_0 \in T$, where $\be_t(f) = f \ci \psi_t$ for all $f \in C(M)$, $t \in T$.
\end{thm}

The key continuity result which we need in order to prove Theorem \ref{t:convdiff} is contained in the following proposition. For the statement, 
we recall that the operator semi-norm $\| \cd \|_{L_{C(M)}}$ is defined on the linear endomorphisms of $C^\infty(M)$ that are bounded with respect to the semi-norm $L_{C(M)}\colon  C^\infty(M) \to [0,\infty)$, see \eqref{eq:opesemi}. 

For each coordinate system $(W,z)$ on $M$ we have the smooth map $g_z \colon  W \to \T{GL}_{\dim(M)}^+(\rr)$ defined by $(g_z)_{ij} = \inn{\frac{\pa}{\pa z_i}, \frac{\pa}{\pa z_j}}$. This yields a unitary local trivialization $\Phi_W \colon  TW \to W \ti \cc^{\dim(M)}$ defined by
\[
\Phi_W\left( \frac{\pa}{\pa z_i}\right) := g_z^{1/2}(e_i)  \q \T{for all } i \in \{1,2,\ldots,\dim(M) \} .
\]

\begin{prop}\label{p:whitcont}
Suppose that $\psi \colon  T \to \T{Diff}(M)$ is continuous with respect to the Whitney $C^1$-topology. Then the map $T \to [0,\infty)$ defined by $t \mapsto \| \be_t \|_{L_{C(M)}}$ is continuous. 
\end{prop}
\begin{proof}
Let $t_0 \in T$ and $\ep > 0$ be given. It suffices to find an open neighbourhood $\Om$ of $t_0$ such that
\begin{equation}\label{eq:pointineq}
\big| \| d( f \ci \psi_t)( \psi_t^{-1}(q) ) \|_\infty - \| d(f \ci \psi_{t_0})( \psi_{t_0}^{-1}(q)) \|_\infty \big| \leq \ep 
\end{equation}
for all $t \in \Om$, $q \in M$ and $f \in \B B_1( L_{C(M)})$. Indeed, if the above inequality holds, then
\[
\begin{split}
 \big| L_{C(M)}( \be_t(f)) - L_{C(M)}(\be_{t_0}(f)) \big|  
 &= \Big| \sup_{q \in M} \| d(f \ci \psi_t)(\psi_t^{-1}(q)) \|_\infty - \sup_{q \in M} \| d(f \ci \psi_{t_0})(\psi_{t_0}^{-1}(q)) \|_\infty \Big|\\
&\leq \ep 
\end{split}
\]
for all $t \in \Om$ and all $f \in \B B_1(L_{C(M)})$, and hence
\[
\begin{split}
\big| \| \be_t \|_{L_{C(M)}} - \| \be_{t_0} \|_{L_{C(M)}} \big|  = \Big| \sup_{f \in \B B_1(L_{C(M)})} L_{C(M)}(\be_t(f) ) - \sup_{f \in \B B_1(L_{C(M)})} L_{C(M)}(\be_{t_0}(f) ) \Big|
\leq \ep 
\end{split}
\]
for all $t \in \Om$. {
At first glance, the evaluation at $\psi_t^{-1}(q)$ may seem strange, but it serves the purpose of making some of the estimates below a bit more manageable.}
So we focus on establishing the inequality in \eqref{eq:pointineq}. In fact, since $M$ is compact we only need to establish the inequality in \eqref{eq:pointineq} locally. More precisely, we fix a compact set $L \su M$ which is contained in a coordinate chart $(V,y)$, and  are then going to find an open neighbourhood $\Om$ of $t_0$ such that
\[
\big| \| d( f \ci \psi_t)( \psi_t^{-1}(q) ) \|_\infty - \| d(f \ci \psi_{t_0})( \psi_{t_0}^{-1}(q)) \|_\infty \big| \leq \ep 
\]
for all $t \in \Om$, $q \in L$ and $f \in \B B_1( L_{C(M)})$.\\
We start by choosing an open set $V' \su V$ such that $L \su V'$ and $\T{cl}(V') \su V$. Now consider the sets
\[
U' := \psi_{t_0}^{-1}(V') \q U := \psi_{t_0}^{-1}(V) \, \, \T{and} \, \, \, K := \T{cl}(U') = \psi_{t_0}^{-1}(\T{cl}(V')),
\]
and remark that $K \su U$ and $\psi_{t_0}(K) \su V$. We put $x := y \ci \psi_{t_0} \colon  U \to \rr^{\dim(M)}$ so that $(U,x)$ is a coordinate chart on $M$. We apply the notation $(U',x')$ and $(V',y')$ for the coordinate charts obtained from $(U,x)$ and $(V,y)$ by restriction.\\
Using the compactness of the sets $K$ and $L$, together with the continuity of the involved maps, we may choose a constant $C>0$ such that
\[
\| g_x^{-1/2}(p) \|_\infty \, , \, \, \| g_y^{1/2}(q) \|_\infty \, , \, \, \| J_{y,x}(\psi_{t_0})(\psi_{t_0}^{-1}(q)) \|_\infty \leq C 
\]
for all $p \in K$ and $q \in L$. Moreover, since continuous maps on compact metric spaces are uniformly continuous there exists a $\delta>0$ such that
\[
\| J_{y,x}(\psi_{t_0})(p) - J_{y,x}(\psi_{t_0})(p')\|_\infty \, , \, \,  \|g_x^{-1/2}(p)- g_x^{-1/2}(p')\|_\infty < \ep/3C^2 
\]
for all $p,p'\in K$ with $d(p,p')<\delta$. \\

Next, put $\ep'=\min\{\delta, \ep/3C^2\}$ and consider the Whitney $C^1$-neighbourhood $W_{\psi_{t_0}}(x, y, K, \ep')$ around $\psi_{t_0}$ and the Whitney $C^1$-neighbourhood $W_{\psi_{t_0}^{-1}}(y', x', L, \ep')$ around $\psi_{t_0}^{-1}$. Since our map $\psi \colon  T \to \T{Diff}(M)$ is assumed continuous with respect to the Whitney $C^1$-topology and $\T{Diff}(M)$ is a topological group, we may choose an open neighbourhood $\Om \su T$ of $t_0 $ such that 
\[
\psi_t\in W_{\psi_{t_0}}(x,y, K, \ep') \q \T{and} \q \psi_t^{-1} \in W_{\psi_{t_0}^{-1}}(y', x',  L, \ep')  
\]
for all $t\in \Omega$. This implies that
\begin{align}
\psi_t^{-1}(L)\su U' \su K \q  &\T{and} \q \psi_t(K) \su V, \\
 d(\psi_t^{-1}(q), \psi_{t_0}^{-1}(q))<\de \q &\T{and} 
\q \big\| J_{y,x}(\psi_{t})( p) - J_{y,x}(\psi_{t_0})( p) \big\|_\infty < \frac{\ep}{3 C^2}
\end{align}
for all $t\in \Omega$, $p \in K$ and $q \in L$.

For each $q \in L$ and each $t \in \Om$, we recall that 
\[
\begin{split}
d \psi_t(\psi_t^{-1}(q))\left( \frac{\pa}{\pa x_j}\Big|_{\psi_t^{-1}(q)} \right)
& = \sum_{i = 1}^{\dim(M)} \frac{\pa}{\pa y_i}\Big|_q \cd \frac{\pa (y_i \ci \psi_t)}{\pa x_j}\Big|_{\psi_t^{-1}(q)} \\
& = \sum_{i = 1}^{\dim(M)} \frac{\pa}{\pa y_i}\Big|_q \cd \big( J_{y,x}(\psi_t)(\psi_t^{-1}(q)) \big)_{ij},
\end{split}
\]
for all $j \in \{1,2,\ldots,\dim(M)\}$. We thus have the identity
\[
\Phi_V(q) d \psi_t(\psi_t^{-1}(q)) \Phi_U^{-1}(\psi_t^{-1}(q))
= g_y^{1/2}(q) \cd J_{y,x}( \psi_t)(\psi_t^{-1}(q)) \cd g_x^{-1/2}(\psi_t^{-1}(q)) ,
\]
where the operator on the right hand side is expressed as an element in the matrix algebra $\B M_{\dim(M)}(\cc)$. As we will see in a moment, the following estimate, holding for each $q \in L$ and each $t \in \Om$, is the key to obtain \eqref{eq:pointineq}:
\[
\begin{split}
& \hspace{0.55cm} \big\| \Phi_V(q) d \psi_t(\psi_t^{-1}(q)) \Phi_U^{-1}(\psi_t^{-1}(q)) - \Phi_V(q) d \psi_{t_0}(\psi_{t_0}^{-1}(q)) \Phi_U^{-1}(\psi_{t_0}^{-1}(q)) \big\|_\infty \\
&  \leq \big\| g_y^{1/2}(q) \big\|_\infty \cd \big\| J_{y,x}( \psi_t)(\psi_t^{-1}(q)) \cd g_x^{-1/2}(\psi_t^{-1}(q)) 
 - J_{y,x}( \psi_{t_0})(\psi_{t_0}^{-1}(q)) \cd g_x^{-1/2}(\psi_{t_0}^{-1}(q)) \big\|_\infty \\
& \leq C \cd \| J_{y,x}( \psi_t)(\psi_t^{-1}(q)) - J_{y,x}(\psi_{t_0})(\psi_{t_0}^{-1}(q)) \|_\infty \cd \| g_x^{-1/2}(\psi_t^{-1}(q)) \|_\infty  \\
&  \hspace{0.4cm} + C \cd \| J_{y,x}( \psi_{t_0}) (\psi_{t_0}^{-1}(q)) \|_\infty \cd \| g_x^{-1/2}(\psi_t^{-1}(q)) - g_x^{-1/2}(\psi_{t_0}^{-1}(q))\|_\infty \\
& < C^2 \cd \| J_{y,x}( \psi_t)(\psi_t^{-1}(q)) - J_{y,x}(\psi_{t_0})(\psi_t^{-1}(q)) \|_\infty \\
&   \hspace{0.4cm}+ C^2 \cd \| J_{y,x}( \psi_{t_0})(\psi_t^{-1}(q)) - J_{y,x}(\psi_{t_0})(\psi_{t_0}^{-1}(q)) \|_\infty + \frac{\ep}{3} \\
& < \ep .
\end{split}
\]
We now consider an $f \in C^\infty(M)$ with $L_{C(M)}(f):=\sup_{p\in M} \|df(p)\|_{\infty} \leq 1$. We then see that
\[
\begin{split}
& \hspace{1.36cm} \Big| \| d(f \ci \psi_t) (\psi_t^{-1}(q)) \|_\infty -  \| d(f \ci \psi_{t_0}) (\psi_{t_0}^{-1}(q)) \|_\infty \Big| \\
& \q = \Big| \| df(q) d\psi_t( \psi_t^{-1}(q)) \|_\infty -  \| df(q) d\psi_{t_0} (\psi_{t_0}^{-1}(q)) \|_\infty \Big|  \\
& \q = \Big| \| df(q) d\psi_t( \psi_t^{-1}(q)) \Phi_U^{-1}(\psi_t^{-1}(q)) \|_\infty -  \| df(q) d\psi_{t_0} (\psi_{t_0}^{-1}(q)) \Phi_U^{-1}(\psi_{t_0}^{-1}(q))\|_\infty \Big|  \\
& \q \leq \| df(q) \|_\infty \cd \big\| d\psi_t( \psi_t^{-1}(q)) \Phi_U^{-1}(\psi_t^{-1}(q))  -  d\psi_{t_0} (\psi_{t_0}^{-1}(q)) \Phi_U^{-1}(\psi_{t_0}^{-1}(q))\big\|_\infty \\
& \q < \ep 
\end{split}
\]
for all $q \in L$ and all $t \in \Om$. This proves the proposition. 
\end{proof}

We are now ready to provide the proof of the main theorem of this {section}:

\begin{proof}[Proof of Theorem \ref{t:convdiff}]
We need to verify the conditions in Theorem \ref{t:crosconv}. We start by noting that the map $T \to C(M)$ given by $t \mapsto \be_t(f) = f \circ \psi_t$ is continuous whenever $f \in C(M)$ is fixed. To see this, one may apply the argument given in the proof of Corollary \ref{c:metisomet} and we will therefore not repeat the details here. 

We now verify the conditions $(1)$, $(2)$ and $(3)$ from Theorem \ref{t:crosconv}. Condition $(1)$ is already proved in Proposition \ref{p:quametrie}. For condition $(2)$ and $(3)$ we fix an $n \in \zz$ and a $t_0 \in T$. Since $\T{Diff}(M)$ is a topological group with respect to the Whitney $C^1$-topology we know that the maps $T \to \T{Diff}(M)$ given by $t \mapsto \psi_t^{-n} \ci \psi_{t_0}^n$ and $t \mapsto \psi_{t_0}^{-n} \ci \psi_t^n$ are continuous. We have the identities 
\[
\begin{split}
\be_{t_0,t}^n = \be_{t_0}^n \ci \be_t^{-n} = \be_{\psi_t^{-n} \ci \psi_{t_0}^n} \q \T{and} \q
\be_{t,t_0}^n = \be_t^n \ci \be_{t_0}^{-n} = \be_{\psi_{t_0}^{-n} \ci \psi_t^n} ,
\end{split}
\]
which together with Proposition \ref{p:whitcont} imply that the maps 
\[
t \mapsto \| \be_{t_0,t}^n \|_{L_{C(M)}} \q \T{and} \q t \mapsto \| \be_{t,t_0}^n \|_{L_{C(M)}}
\]
are continuous. By Lemma \ref{lem:operator-semi-norm-cont}, this entails that condition $(3)$ holds. Moreover, using that $\be_{t_0,t_0}^n = \T{id} : C(M) \to C(M)$ we may, for each $\ep > 0$, find an open neighbourhood $\Om$ of $t_0 \in T$ such that for all $f \in C^\infty(M)$ and all $t \in \Om$ we have
\[
\begin{split}
 L_{C(M)}( \be_t^n(f)) & =L_{C(M)}(\be_t^n\circ \be_{t_0}^{-n}\circ \be_{t_0}^n(f))
\leq \|\be_{t,t_0}^n\|_{L_{C(M)}}  L_{C(M)}(\be_{t_0}^n(f))  \\ & \leq (1 + \ep) \cd L_{C(M)}(\be_{t_0}^n(f)) .
\end{split}
\]
This shows that condition $(2)$ also holds and we have proved the theorem.
\end{proof}


\providecommand{\bysame}{\leavevmode\hbox to3em{\hrulefill}\thinspace}
\providecommand{\MR}{\relax\ifhmode\unskip\space\fi MR }
\providecommand{\MRhref}[2]{%
  \href{http://www.ams.org/mathscinet-getitem?mr=#1}{#2}
}
\providecommand{\href}[2]{#2}

\end{document}